\documentclass[11pt,a4paper]{article}

\usepackage{epsf,epsfig,amsfonts,amsgen,amsmath,amstext,amsbsy,amsopn,amsthm}
\usepackage{amsmath}
\usepackage{amsfonts,amsthm,amssymb,bm}
\usepackage{amsfonts}
\usepackage{graphics}
\usepackage{latexsym,bm}
\usepackage{amsfonts,amsthm,amssymb,bbding}
\usepackage{indentfirst}
\usepackage{graphicx}
\usepackage{color}

\usepackage[colorlinks=true,anchorcolor=blue,filecolor=blue,linkcolor=blue,urlcolor=blue,citecolor=red]{hyperref}
\usepackage{float}
\usepackage{tikz,enumerate}
\usepackage[margin=2.5cm]{geometry}

\newtheorem{thm}{Theorem}[section]
\newtheorem{prob}[thm]{Problem}

\newtheorem{lem}[thm]{Lemma}
\newtheorem{cor}[thm]{Corollary}
\newtheorem{pro}[thm]{Proposition}

\newtheorem{remark}[thm]{Remark}

\newtheorem{conj}[thm]{Conjecture}
\newtheorem{claim}{Claim}[section]
\newtheorem{definition}{Definition}[section]

\addtocounter{section}{0}

\usepackage{enumitem}  %%for enumerate environment

\begin{document}
\title{The spectral Tur\'{a}n problem: Characterizing\\
 spectral-consistent graphs}

\author{
Longfei Fang\thanks{School of Mathematics, East China University of Science and Technology, Shanghai 200237, China. Email: \url{lffang@chzu.edu.cn}.}
\and
Sergey Goryainov\thanks{School of Mathematical Sciences, Hebei International Joint Research Center for Mathematics and Interdisciplinary Science, Hebei Key Laboratory of Computational Mathematics
and Applications, Hebei Workstation for Foreign Academicians, Hebei Normal University, Shijiazhuang 050024, P.R. China. Email: \url{sergey.goryainov3@gmail.com}.}
\and
Denis Krotov\thanks{School of Mathematical Sciences, Hebei Key Laboratory of Computational Mathematics and Applications, Hebei Normal University, Shijiazhuang 050024, P. R. China, and
Sobolev Institute of Mathematics, Novosibirsk 630090, Russia.
Email: \url{dk@ieee.org}.}
\and
Huiqiu Lin\thanks{School of Mathematics, East China University of Science and Technology, Shanghai 200237, China.  Email: \url{huiqiulin@126.com}.}
\and
Mingqing Zhai\thanks{Corresponding author.
School of Mathematics and Statistics, Nanjing University of Science and
Technology, Nanjing, Jiangsu 210094, China.
 Email: \url{mqzhai@njust.edu.cn}.}
}

\date{\today}
\maketitle

\begin{abstract}
Let ${\rm EX}(n,H)$ and ${\rm SPEX}(n,H)$
denote the families of $H$-free graphs of order $n$ with maximum size and maximum spectral radius, respectively.
A graph $H$ is called spectral-consistent if ${\rm SPEX}(n,H)\subseteq {\rm EX}(n,H)$ holds for all sufficiently large $n$.
Cioab\u{a}, Desai, and Tait put forward the following conjecture:
Let $H$ be a graph for which every graph in ${\rm EX}(n,H)$ is a Tur\'{a}n graph plus $O(1)$ edges.
Then $H$ is spectral-consistent.
Wang, Kang, and Xue settled this conjecture.
Recently, Liu and Ning raised a general problem in spectral extremal graph theory:
Characterize all graphs that are spectral-consistent.

In this paper, we establish that for any finite graph $H$, if its decomposition family is matching-good, then $H$ is necessarily spectral-consistent.
Notably, this structural condition is strictly weaker than the condition for spectral-consistency established by Wang, Kang, and Xue in their earlier work,
thereby broadening the class of graphs known to satisfy the spectral-consistency property.
Our main result enables us to fully characterize  spectral consistency for several important families of forbidden graphs $H$,
including generalized color-critical graphs,
odd-balloonings of nontrivial trees and complete bipartite graphs $K_{s,t}$ with $t\geq s\geq 2$ and $(s,t)\neq(2,3)$, as well as edge blow-ups of non-bipartite graphs and even cycles.
Furthermore, we present a streamlined proof of the Chen--Lei--Li theorem on the
spectral-consistency of disjoint unions of suspensions of unions of
color-critical graphs.
Finally, we propose an open problem to motivate future research in this area.
\end{abstract}

\begin{flushleft}
\textbf{Keywords:} spectral-consistent; generalized color-critical graph; odd-ballooning; edge blow-up
\end{flushleft}
\textbf{AMS Classification:} 05C35; 05C50

\section{Introduction}

A graph family $\mathcal{H}$ is called \emph{finite}
if $|\mathcal{H}|<\infty$ and $|V(H)|<\infty$ for every $H\in\mathcal{H}$.
We call a graph \emph{$\mathcal{H}$-free}
if it does not contain any copy of $H\in \mathcal{H}$ as a subgraph.
The Tur\'{a}n number of $\mathcal{H}$, denoted by ${\rm ex}(n,\mathcal{H})$,
is the maximum number of edges of any $n$-vertex $\mathcal{H}$-free graph.
Let ${\rm EX}(n,\mathcal{H})$ denote the family of $\mathcal{H}$-free graphs
with  $n$ vertices and ${\rm ex}(n,\mathcal{H})$ edges.
For convenience, we use ${\rm ex}(n,H)$ (resp. ${\rm EX}(n,H)$) instead of ${\rm ex}(n,\mathcal{H})$ (resp. ${\rm EX}(n,\mathcal{H})$)
when $\mathcal{H}=\{H\}$.
Given a graph family  $\mathcal{H}$ with  $\min_{H\in \mathcal{H}}\chi(H)=p+1\geq 3$,
where $\chi(H)$ denotes the chromatic number of $H$,
the famous Erd\H{o}s--Stone--Simonovits theorem \cite{Erdos-1966,ES1946}
shows that
$${\rm ex}(n,\mathcal{H})=\frac{p-1}{p}\binom{n}{2}+o(n^2).$$
Since then, investigating the Tur\'{a}n number for various graph families $\mathcal{H}$
has emerged as a central focus within extremal graph theory.
For further details on this topic,
we refer readers to \cite{EFGG1995,EG1959,FS1975,FG2015,FS2013} and the references therein.

Given a graph $G$, let $A(G)$ denote its adjacency matrix and $\rho(G)$ the spectral radius of $A(G)$.
Define ${\rm spex}(n,\mathcal{H})$ as
the maximum spectral radius of any $n$-vertex $\mathcal{H}$-free graph.
A graph $G$ with $n$ vertices is said to be \textit{extremal} for ${\rm spex}(n,\mathcal{H})$
if $G$ is $\mathcal{H}$-free and $\rho(G)={\rm spex}(n,\mathcal{H})$.
Denote by ${\rm SPEX}(n,\mathcal{H})$ the family of extremal graphs for  ${\rm spex}(n,\mathcal{H})$.
When $\mathcal{H}=\{H\}$, we simplify the notation by writing
${\rm spex}(n,H)$ and ${\rm SPEX}(n,H)$  for ${\rm spex}(n,\mathcal{H})$ and ${\rm SPEX}(n,\mathcal{H})$,
respectively.
A finite graph family $\mathcal{H}$ is said to be \textit{spectral-consistent}
if ${\rm SPEX}(n,\mathcal{H})\subseteq {\rm EX}(n,\mathcal{H})$ for sufficiently large $n$.
In 2020, Cioab\u{a}, Feng, Tait, and Zhang \cite{CFTZ2020} showed that the friendship graph $F_k$ is spectral-consistent.
Recently, Liu and Ning \cite{Liu2023+} posed the following problem:
Characterize all graphs $H$ that are spectral-consistent.
By replacing $H$ with a graph family $\mathcal{H}$,
 we arrive at a fundamental problem in spectral extremal graph theory.

\begin{prob}\label{prob1.1}%\emph{(\cite{Liu2023+})}
Characterize all graph families $\mathcal{H}$ that are spectral-consistent.
\end{prob}

Related spectral extremal results have also been obtained
for forbidden graphs whose orders grow with the order
of the host graph, including
$C_n$ \cite{Fiedler2010}, $C_{n-1}$ \cite{Ge2020},
$C_{\ell}$ ($n-c_1\sqrt{n}\leq\ell\leq n$) \cite{Li2023},
and any $n$-vertex graph $H$ with $\delta(H)\geq1$ and $\Delta(H)\leq {\sqrt{n}}/{40}$ \cite{Liu2023+}.
In addition to the results for large graphs,
spectral consistency has also been established
for certain finite graphs, such as friendship
graphs \cite{CFTZ2020}.
On the other hand, Cioab\u{a}, Desai, and
Tait \cite{Cioaba2} proved that
$\mathrm{SPEX}(n,W_{2k+1})
\cap\mathrm{EX}(n,W_{2k+1})=\varnothing$
for $k=7$ or $k\geq 9$ and sufficiently large $n$.
Building on these findings, they proposed the following conjecture for a finite graph $H$.

\begin{conj}\emph{(\cite{Cioaba2})}\label{conj2.1}
Let $H$ be any graph such that the graphs in ${\rm EX}(n,H)$ can be obtained from the $p$-partite Tur\'{a}n graph $T_p(n)$ by adding $O(1)$ edges.
Then, $H$ is spectral-consistent.
\end{conj}

Recently, Wang, Kang, and Xue \cite{WANG-2023} confirmed Conjecture \ref{conj2.1}
by showing that if $\mathrm{ex}(n,H)=e(T_p(n))+O(1)$, then $H$ is spectral-consistent.

Let $p(\mathcal{H})=\min_{H\in \mathcal{H}}\chi(H)-1$, and let $E_t$ be the empty graph of order $t$.
Let $F\nabla H$ and $F\cup H$ denote the join and the union of $F$ and $H$, respectively.
The following concept, introduced by
Simonovits \cite{Simonovits1974},
has proven to be extremely valuable in solving extremal problems.

\begin{definition}\emph{(\cite{Simonovits1974})}\label{definition-1.1}
Given a graph family $\mathcal{H}$ with $p(\mathcal{H})=p\geq 2$,
the \textbf{decomposition family} $\mathcal{M}(\mathcal{H})$ consists of minimal graphs $M$ for which
there exist some $H\in \mathcal{H}$ and some constant $t=t(\mathcal{H})$ such that
$H\subseteq (M\cup E_t)\nabla T_{p-1}((p-1)t)$.
\end{definition}

For a finite graph family $\mathcal{H}$,
we shall repeatedly use the following elementary consequence of the
minimality in the definition of the decomposition family. If
$M\in\mathcal{M}(\mathcal{H})$, then
$|V(M)|\leq \max_{H\in\mathcal{H}}|V(H)|$.
Indeed, since $M\in\mathcal{M}(\mathcal{H})$,
by the definition of the decomposition family,
there exist $H\in\mathcal{H}$ and a positive
integer $s$ such that
$H\subseteq (M\cup E_s)\nabla T_{p-1}((p-1)s)$.
By the inclusion-minimality of $M$, every vertex of $M$ must be used
by the copy of $H$; otherwise, deleting an unused vertex
would give a proper subgraph of $M$ with the same property.
Consequently, $|V(M)|\leq |V(H)|$.
A direct consequence of this definition is the following proposition.

\begin{pro}\label{prop1.1}
Let $\mathcal{H}$ be a graph family with $p(\mathcal{H})=p\geq 2$.
Then, there must exist bipartite members in $\mathcal{M}(\mathcal{H})$.
\end{pro}

\begin{proof}
Recall that $p(\mathcal{H})=\min_{H\in\mathcal{H}}\chi(H)-1$.
Choose $H\in\mathcal{H}$ with $\chi(H)=p+1$, and fix a proper
$(p+1)$-coloring of $H$ with color classes
$V_1,\ldots,V_{p+1}$. Let
$M_0=H[V_1\cup V_2]$. Then $M_0$ is bipartite. Moreover, the
remaining $p-1$ color classes can be embedded into the $p-1$ parts of
$T_{p-1}((p-1)|H|)$, and hence
$H\subseteq M_0\nabla T_{p-1}((p-1)|H|)$.
By the definition of $\mathcal{M}(\mathcal{H})$, there exists
$M\in\mathcal{M}(\mathcal{H})$ with $M\subseteq M_0$. Therefore $M$
is bipartite.
\end{proof}

In particular, if $\mathcal{H}=\{H\}$, then the decomposition family of $H$ is denoted by $\mathcal{M}(H)$.
For an integer $k$, let $K_{k}$, $S_k$, and $C_k$ denote the complete graph, the star, and the cycle on $k$ vertices, respectively.
Additionally,  let $M_{k}$ be the disjoint union of $k$ copies of $K_2$.
Recently,
Fang, Tait, and Zhai \cite{Fang-2025} provided a characterization of the finite graph families $\mathcal{H}$ whose decomposition families $\mathcal{M}(\mathcal{H})$ contain both a matching and a star.

\begin{thm}\label{THM1.1A}\emph{(\cite{Fang-2025})}
Let $\mathcal{H}$ be a finite graph family
with $p(\mathcal{H})=p\geq 2$ and $\phi=\max\{|H|: H\in \mathcal{H}\}$.
Then, for sufficiently large $n$,
the following statements are equivalent:

\vspace{1mm}
{\rm (i)}  There exist two integers $\nu\leq\lfloor\frac{\phi}{2}\rfloor$
and $\Delta\leq \phi$ such that $M_{\nu},
S_{\Delta+1}\in\mathcal{M}(\mathcal{H})$;

\vspace{1mm}
{\rm (ii)}  $e(T_{p}(n))\leq {\rm ex}(n,\mathcal{H})<e(T_{p}(n))+\lfloor \frac{n}{2p} \rfloor$;

\vspace{1mm}
{\rm (iii)}  Every graph in ${\rm EX}(n,\mathcal{H})$
is obtained from $T_{p}(n)$ by adding and deleting $O(1)$ edges.
%\end{enumerate}
\end{thm}

By combining Theorem \ref{THM1.1A} with the result of Wang, Kang, and Xue \cite{WANG-2023},
we conclude that if $\mathcal{M}(H)$ contains both a matching and a star, then
$H$ is spectral-consistent.
Further spectral-consistency results have been established
for certain graph families whose decomposition families
contain a matching, including intersecting cliques
\cite{DKL2022}, $k$ edge-disjoint triangles
\cite{Lin2022}, and edge blow-up of star forests
\cite{WANG-2024}.

\subsection{Definitions and notations} \label{sub1.2}

To state our main theorems, we first introduce some definitions and notations.
Given a graph $G$, we use the following standard notation:
the vertex set $V(G)$, the edge set $E(G)$, the order $|G|$, the size $e(G)$,
the maximum degree $\Delta(G)$, and the minimum degree $\delta(G)$, respectively.
Given two disjoint subsets $X,Y\subseteq V(G)$,
let $G[X]$ be the subgraph induced by $X$, and write $G-X=G[V(G)-X]$.
Let $G[X,Y]$ be the bipartite subgraph on the vertex set $X\cup Y$,
consisting of all edges with one endpoint in $X$ and the other in $Y$.
For brevity, we write $e(X)=e(G[X])$ and $e(X,Y)=e(G[X,Y])$, respectively.
We also need the parameters $\beta(\mathcal{M}(\mathcal{H}))$ and $\gamma(\mathcal{M}(\mathcal{H}))$ defined with respect to $\mathcal{M}(\mathcal{H})$.
A \emph{covering} of a graph is a set of vertices that intersects all edges.
We define
            $$\beta(\mathcal{M}(\mathcal{H}))=\min\{\beta(M)~|~M\in \mathcal{M}(\mathcal{H})\},$$
where $\beta(M)$ represents the number of vertices in a minimum covering of $M$.

Similarly, an \emph{independent covering} of a bipartite graph is an independent set that intersects all edges.
In view of Proposition \ref{prop1.1}, we define
            $$\gamma(\mathcal{M}(\mathcal{H}))=\min\{\gamma(M)~|~M\in \mathcal{M}(\mathcal{H})~~\text{is bipartite} \},$$
where $\gamma(M)$ denotes the number of vertices in a minimum independent covering of $M$.
When there is no ambiguity,
we write $\beta$ and $\gamma$ to represent
$\beta(\mathcal{M}(\mathcal{H}))$ and $\gamma(\mathcal{M}(\mathcal{H}))$, respectively.
By the definition of $\gamma$,
there exists a bipartite graph
$M_0\in \mathcal{M}(\mathcal{H})$ such that $\gamma(M_0)=\gamma$.
Combining this with the definition of $\beta$,
we obtain
\begin{align}\label{align-001}
\beta\leq\beta(M_0)\leq \gamma(M_0)=\gamma.
\end{align}
If $\beta=\gamma$,
we set $\mathcal{B}(\mathcal{H})=\{K_{\gamma}\}$. Otherwise, define
   $$\mathcal{B}(\mathcal{H})=\{M[S]~|~M\in \mathcal{M}(\mathcal{H}),~S~\text{is a covering of}~M~\text{with}~|S|<\gamma\}.$$

\subsection{Main results}

\begin{definition}\label{definition-1.2}
The decomposition family $\mathcal{M}(\mathcal{H})$ of
a finite graph family $\mathcal{H}$ is said to be  \textbf{matching-good}
if the following conditions are satisfied:

\vspace{1mm}
{\rm (i)}  $p(\mathcal{H})=p\geq 2$, and $\mathcal{M}(\mathcal{H})$ contains a matching;

\vspace{1mm}
{\rm (ii)}
For every sufficiently large $n$,
there exists a graph $G^\star\in{\rm EX}(n,\mathcal{H})$
such that $G^\star=H_1\nabla H_2$,
where $H_1\in{\rm EX}(\gamma-1,\mathcal{B}(\mathcal{H}))$
and $H_2$ differs from $T_p(n-\gamma+1)$
by at most $C$ edges, for some constant
$C=C(\mathcal{H})$ independent of $n$.
\end{definition}

Our first main theorem is stated as follows:

\begin{thm}\label{thm1.1}
Let $\mathcal{H}$ be a finite graph family whose decomposition family $\mathcal{M}(\mathcal{H})$ is matching-good.
Then, $\mathcal{H}$ is spectral-consistent.
\end{thm}

By Theorem \ref{THM1.1A}, $\mathrm{ex}(n,H)=e(T_p(n))+O(1)$ if and only if $\mathcal{M}(H)$ contains both a star and a matching as its members.
Let $H$ be a graph satisfying the hypothesis of Conjecture \ref{conj2.1}.
This case corresponds to $\gamma=1$ in Definition \ref{definition-1.2},
which implies that $\mathcal{M}(H)$ is matching-good.
Consequently, Theorem \ref{thm1.1} strengthens
Conjecture \ref{conj2.1} and offers a partial answer to Problem \ref{prob1.1}.
In the following, we present some important applications of Theorem \ref{thm1.1}.

\subsubsection{Vertex-disjoint union of graphs with matching-good decomposition families}

Ni, Wang, and Kang \cite{NWK2023+} proved that if
$H$ is a finite graph such that $\mathrm{ex}(n,H)=e(T_p(n))+O(1)$ for sufficiently large $n$,
then $\mathcal{M}(tH)$ is matching-good and $tH$ is spectral-consistent.
For the sake of applications, we use a different method to generalize their result as follows.

\begin{thm}\label{thm1.9A}
Let $t\geq 1$, $p\geq 2$ be integers, and for each $i\in [t]=\{1,2,\ldots,t\}$, let $H_i$ be a finite graph such that $\mathrm{ex}(n,H_i)=e(T_p(n))+O(1)$.
Then, $\mathcal{M}(\cup_{i\in [t]}H_i)$ is matching-good.
\end{thm}

For $p\geq 2$, $t\geq 1$ and $k_1\geq k_2\geq \cdots \geq k_t\geq 1$, let
\begin{equation}\label{equation-0A}
\begin{array}{ll}
H=\cup_{i\in [t]}H_i~~\text{with}~~H_i:=K_1\nabla (\bigcup_{j\in [k_i]}H_{i_j}),
~\chi(H_{i_j})=p,~i\in [t],~j\in [k_i],
\end{array}
\end{equation}
where $H_{i_j}$ is a color-critical graph for every $i\in [t]$ and $j\in [k_i]$.
Hou, Li and Zeng \cite{Hou-2024} showed that for sufficiently large $n$ and each $i\in [t]$,
$$\mathrm{ex}(n,H_i)=e\big(T_p(n)\big)+\left\{
                                       \begin{array}{ll}
                                        k_i^2-k_i & \hbox{if $k_i$ is odd,} \\
                                        k_i^2-\frac{3}{2}k_i  & \hbox{if $k_i$ is even.}
                                       \end{array}
                                     \right.
$$
Combining this with Theorem \ref{thm1.1} and Theorem \ref{thm1.9A}, we can directly obtain the following result,
which was recently given by Chen, Lei, and Li \cite{Chen2025}.

\begin{thm}\label{thm1.10A}\emph{(\cite{Chen2025})}
Let $H$ be any graph defined in \eqref{equation-0A}.
Then, $H$ is spectral-consistent.
\end{thm}

\subsubsection{Generalized color-critical graph family}

Recall that $p(\mathcal{H})=\min_{H\in \mathcal{H}}\chi(H)-1$.
Inspired by the work of Simonovits in \cite{Simonovits1974,Simonovits1999},
we introduce the notion of generalized color-critical graph families.

\begin{definition}\label{definition-1.3}
Let $q\geq 1$, and let $\mathcal{H}$ be a finite graph family with $p(\mathcal{H})=p\geq 2$
and $|V(H)|\geq q-1$ for every $H\in\mathcal{H}$.
The family $\mathcal{H}$ is said to be \textbf{$q$-color-critical}
if it satisfies the following conditions:

\vspace{1mm}
{\rm (i)}  For any $H\in \mathcal{H}$ and any $(q-1)$-subset $S\subseteq V(H)$, we have $\chi(H-S)\geq p+1$;

\vspace{1mm}
{\rm (ii)}  There exists a graph $H\in \mathcal{H}$ that admits a
 $(p+1)$-coloring, where the subgraph induced by the first two color classes consists of $q$ independent edges and possibly some isolated vertices.
\end{definition}

By Definition \ref{definition-1.3}, a graph $H$ is 1-color-critical if and only if it contains an edge whose deletion reduces its chromatic number.
In other words, the decomposition family of $H$ consists solely of an edge.
It is clear that odd cycles, book graphs, even wheels, and complete graphs are all 1-color-critical graphs.
Furthermore, condition {\rm (ii)} of
Definition~\ref{definition-1.3} implies that
$\mathcal{M}(\mathcal{H})$ contains a matching $M_{q'}$ for some
$1\leq q'\leq q$. We shall show in the proof of
Theorem~\ref{TH1.6} that, for a $q$-color-critical graph family,
one necessarily has $q'=q$.
Let $H(n,p,q)=K_{q-1}\nabla T_{p}(n-q+1)$, where $H(n,p,1)=T_{p}(n)$.
The following theorem, known as the generalized color-critical theorem, was obtained by Simonovits
(see \cite[Theorem 2.2]{Simonovits1974} and \cite[Theorem 1.7]{Simonovits1999}).

\begin{thm}\emph{(\cite{Simonovits1974,Simonovits1999})}\label{TH1.3}
Let $\mathcal{H}$ be a $q$-color-critical graph family
with $p(\mathcal{H})=p\geq 2$.
Then, for sufficiently large $n$, we have $\mathrm{EX}(n,\mathcal{H})=\{H(n,p,q)\}$.
\end{thm}

The following graphs are known to be generalized color-critical:
the disjoint union of 1-color-critical graphs with the same chromatic number
(see \cite[p. 356]{Simonovits1974});
the dodecahedron graph $D^{20}$ (see \cite[p. 288]{Simonovits1999});
the Petersen graph $P^{10}$ (see \cite[p. 289]{Simonovits1999});
and the Kneser graph $K(t,2)$ for $t\geq 6$ (see \cite[p. 22]{Ma2025}).
Using Theorem \ref{TH1.3}, Simonovits obtained the following theorem.

\begin{thm}\label{TH1.4}
For sufficiently large $n$, we have

\vspace{1mm}
{\rm (i)}  \emph{(\cite[Theorem 2.2]{Simonovits1974})} $\mathrm{EX}(n,\cup_{i=1}^{q}H_i)=\{H(n,p,q)\}$,
where $H_i$ is a 1-color-critical graph with $\chi(H_i)=p+1$ for each $i\in [q]$;

\vspace{1mm}
{\rm (ii)}  \emph{(\cite[Theorem 2.4]{Simonovits1974})} $\mathrm{EX}(n,D^{20})=\{H(n,2,6)\}$;

\vspace{1mm}
{\rm (iii)}  \emph{(\cite[Theorem 1.6]{Simonovits1999})} $\mathrm{EX}(n,P^{10})=\{H(n,2,3)\}$.
\end{thm}

We now present a spectral version of Simonovits' generalized color-critical theorem.

\begin{thm}\label{TH1.6}
Let $\mathcal{H}$ be a $q$-color-critical graph family
with $p(\mathcal{H})=p\geq 2$.
Then, $\mathcal{M}(\mathcal{H})$ is matching-good, and consequently,
$\mathrm{SPEX}(n,\mathcal{H})=\{H(n,p,q)\}$ for sufficiently large $n$.
\end{thm}

\begin{remark}
Based on Theorem \ref{TH1.6},
for any single $q$-color-critical graph $H$ with $\chi(H)=p+1\geq 3$,
we have $\mathrm{SPEX}(n,H)=\{H(n,p,q)\}$ for all sufficiently large $n$.
This result can also be derived by combining Theorem \ref{TH1.3}
with a result of Zhang \cite{Zhang2025+}.
Specifically, Theorem \ref{TH1.3} established that $\mathrm{EX}(n,H)=\{H(n,p,q)\}$,
which implies that $H(n,p,q)$ is $H$-free. Meanwhile,
Zhang \cite[Theorem 1.6]{Zhang2025+} showed that
for integers $p\geq 2$, $q\geq 1$, and $m\geq 2q$,
the graph $H^*=(qK_2\cup E_{m-2q})\nabla T_{p-1}(m(p-1))$ satisfies
$\mathrm{SPEX}(n,H^*)=\{H(n,p,q)\}$ when $n$ is sufficiently large.
Notably, $H^*$ is a single $q$-color-critical graph, and $H\subseteq H^*$ provided that $m\geq |H|$.
Combining these findings yields $\mathrm{SPEX}(n,H)=\{H(n,p,q)\}$.
\end{remark}

Theorem \ref{TH1.6} implies some previous results for special cases of $H$,
including 1-color-critical graphs \cite{Guiduli-1996,Nikiforov5,Nikiforov10,ZL2023},
and the disjoint union of 1-color-critical graphs with the same chromatic number
\cite{Fang-2024,Lei-2024,NWK2023}.
By Theorems \ref{TH1.4} and \ref{TH1.6},
we further derive the following corollary.

\begin{cor}
For sufficiently large $n$, we have

\vspace{1mm}
{\rm (i)}  $\mathrm{SPEX}(n,D^{20})=\{H(n,2,6)\}$;

\vspace{1mm}
{\rm (ii)}  $\mathrm{SPEX}(n,P^{10})=\{H(n,2,3)\}$.
\end{cor}

\subsubsection{Odd-ballooning of trees and complete bipartite graphs}

An odd-ballooning of $H$, denoted by $H^{odd}$, is a graph obtained from
$H$ by replacing each edge of $H$ with an odd cycle,
where the new vertices of different odd cycles are distinct.
In particular, the friendship graph $F_k$ is the odd-ballooning of a star $K_{1,k}$, constructed by replacing each edge with a triangle.
Cioab\u{a}, Feng, Tait, and Zhang \cite{CFTZ2020} showed that the friendship graph $F_k$ is spectral-consistent,
and later, Li and Peng \cite{Li} extended this result by proving that the odd-ballooning of any star is spectral-consistent.
Subsequently, Zhai, Liu, and Xue \cite{Zhai2022} characterized the unique extremal graph in $\mathrm{SPEX}(n,F_k)$ for sufficiently large $n$.

\begin{remark}
Throughout this paper,
$H^{odd}$ denotes the odd-ballooning of a graph $H$, where each edge is replaced by an odd cycle of length at least five.
Notably, odd cycles replacing distinct edges may have different lengths.
\end{remark}

Let $\mathcal{B}$ be a family of graphs and let $F$ be a graph.
For integers $p\geq 2$ and $q\geq 1$, let
$\mathcal{H}(n,p,q,\mathcal{B},F)$ denote the family of graphs
obtained from $Q\nabla T_p(n-q+1)$ by embedding a copy of $F$
into one partite class of $T_p(n-q+1)$, where
$Q\in{\rm EX}(q-1,\mathcal{B})$ if $q\geq 2$,
and $Q$ is the graph on zero vertices if  $q=1$.

Recently, Zhu and Chen \cite{Zhu2023} established the following result.
In fact, they provided a stronger version;
however, for the sake of simplicity, we present only the version stated below.

\begin{thm}\label{thm2.2}\emph{(\cite{Zhu2023})}
Let $T$ be a nontrivial tree with partite sets $A$ and $B$ with $|A|\leq |B|$ and $\delta_{T}(A)=k$,
where $\delta_{T}(A)=\min_{v\in A}d_T(v)$.
Then, for sufficiently large $n$, we have
$$\mathcal{H}(n,2,\gamma(\mathcal{M}(T^{odd})),\mathcal{B}(T^{odd}),K_{k-1,k-1})\subseteq \mathrm{EX}(n,T^{odd}).$$
\end{thm}

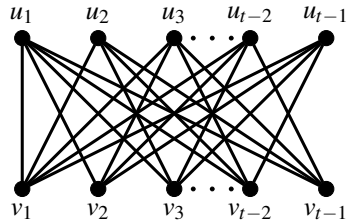
\begin{figure}[!h]
\setlength{\abovecaptionskip}{-12pt}  % 控制标题和图片之间的距离
\setlength{\belowcaptionskip}{0pt}  % 控制标题和下面内容之间的距离
\centering
\begin{tikzpicture}[x=1.00mm, y=1.00mm, inner xsep=0pt, inner ysep=0pt, outer xsep=0pt, outer ysep=0pt]
\path[line width=0mm] (77.43,57.82) rectangle +(6.00,35.15);
\definecolor{L}{rgb}{0,0,0}
\definecolor{F}{rgb}{0,0,0}
\node[circle,fill=F,draw,inner sep=0pt,minimum size=2mm] (u1) at (80,90) {};
\draw(80,93) node[anchor=center]{\fontsize{10}{8.65}\selectfont $u_1$};
\node[circle,fill=F,draw,inner sep=0pt,minimum size=2mm] (u2) at (90,90) {};
\draw(90,93) node[anchor=center]{\fontsize{10}{8.65}\selectfont $u_2$};
\node[circle,fill=F,draw,inner sep=0pt,minimum size=2mm] (u3) at (100,90) {};
\draw(100,93) node[anchor=center]{\fontsize{10}{8.65}\selectfont $u_3$};

\node[circle,fill=F,draw,inner sep=0pt,minimum size=0.5mm] () at (102.5,90) {};
\node[circle,fill=F,draw,inner sep=0pt,minimum size=0.5mm] () at (105,90) {};
\node[circle,fill=F,draw,inner sep=0pt,minimum size=0.5mm] () at (107.5,90) {};

\node[circle,fill=F,draw,inner sep=0pt,minimum size=2mm] (ut2) at (110,90) {};
\draw(110,93) node[anchor=center]{\fontsize{10}{8.65}\selectfont $u_{t-2}$};
\node[circle,fill=F,draw,inner sep=0pt,minimum size=2mm] (ut1) at (120,90) {};
\draw(120,93) node[anchor=center]{\fontsize{10}{8.65}\selectfont $u_{t-1}$};

\node[circle,fill=F,draw,inner sep=0pt,minimum size=2mm] (v1) at (80,70) {};
\draw(80,67) node[anchor=center]{\fontsize{10}{8.65}\selectfont $v_1$};
\node[circle,fill=F,draw,inner sep=0pt,minimum size=2mm] (v2) at (90,70) {};
\draw(90,67) node[anchor=center]{\fontsize{10}{8.65}\selectfont $v_2$};
\node[circle,fill=F,draw,inner sep=0pt,minimum size=2mm] (v3) at (100,70) {};
\draw(100,67) node[anchor=center]{\fontsize{10}{8.65}\selectfont $v_3$};

\node[circle,fill=F,draw,inner sep=0pt,minimum size=0.5mm] () at (102.5,70) {};
\node[circle,fill=F,draw,inner sep=0pt,minimum size=0.5mm] () at (105,70) {};
\node[circle,fill=F,draw,inner sep=0pt,minimum size=0.5mm] () at (107.5,70) {};

\node[circle,fill=F,draw,inner sep=0pt,minimum size=2mm] (vt2) at (110,70) {};
\draw(110,67) node[anchor=center]{\fontsize{10}{8.65}\selectfont $v_{t-2}$};
\node[circle,fill=F,draw,inner sep=0pt,minimum size=2mm] (vt1) at (120,70) {};
\draw(120,67) node[anchor=center]{\fontsize{10}{8.65}\selectfont $v_{t-1}$};

\path[line width=0.40mm, draw=L] (u1) -- (v1);
\path[line width=0.40mm, draw=L] (u2) -- (v1);
\path[line width=0.40mm, draw=L] (u3) -- (v1);
\path[line width=0.40mm, draw=L] (ut2) -- (v1);
\path[line width=0.40mm, draw=L] (ut1) -- (v1);

\path[line width=0.40mm, draw=L] (u1) -- (v2);
\path[line width=0.40mm, draw=L] (u3) -- (v2);
\path[line width=0.40mm, draw=L] (ut2) -- (v2);
\path[line width=0.40mm, draw=L] (ut1) -- (v2);

\path[line width=0.40mm, draw=L] (u1) -- (v3);
\path[line width=0.40mm, draw=L] (u2) -- (v3);
\path[line width=0.40mm, draw=L] (ut2) -- (v3);
\path[line width=0.40mm, draw=L] (ut1) -- (v3);

\path[line width=0.40mm, draw=L] (u1) -- (vt2);
\path[line width=0.40mm, draw=L] (u2) -- (vt2);
\path[line width=0.40mm, draw=L] (u3) -- (vt2);
\path[line width=0.40mm, draw=L] (ut1) -- (vt2);

\path[line width=0.40mm, draw=L] (u1) -- (vt1);
\path[line width=0.40mm, draw=L] (u2) -- (vt1);
\path[line width=0.40mm, draw=L] (u3) -- (vt1);
\path[line width=0.40mm, draw=L] (ut2) -- (vt1);
\end{tikzpicture}%
\caption{The graph $H_{t-1,t-1}$. }{\label{fig1}}
\end{figure}

Denote by $H_{t-1,t-1}$ the graph obtained
from $K_{t-1,t-1}$ by removing a matching of size $t-2$,
namely $\{u_iv_i: 2\leq i\leq t-1\}$ (see Figure \ref{fig1}).
For $t\geq s\geq 2$, let $G_{s,t}$ be the graph
obtained from $K_{s-1}\nabla T_2(n-s+1) $ by embedding
$H_{t-1,t-1}$ into one color class of $T_{2}(n-s+1)$.
For $i=2,3$, let $G'_{i,3}$ be the graph obtained from $K_{i-1}\nabla T_2(n-i+1)$ by embedding a triangle into one color class of $T_2(n-i+1)$.
Let $H'(n,2,2)$ be the graph obtained from $K_1\nabla T_2(n-1)$
by embedding an edge into one color class of $T_2(n-1)$.
Let $K_{2,3}^{(5)}$ be the graph obtained from $K_{2,3}$ by
replacing each edge with a $5$-cycle, with the sets of newly introduced vertices corresponding
to distinct edges being pairwise disjoint.
Yan \cite{Yan-2021} showed that, for sufficiently large $n$,
$G_{2,3},G'_{2,3}\in \mathrm{EX}(n,K_{2,3}^{(5)})$.
This result addresses a special case of the odd-ballooning of $K_{2,3}$, whereas the general case $(t,s)=(3,2)$ is excluded from the results below.
The remaining relevant extremal results are summarized as follows.

\begin{thm}\label{thm2.3}
Let $t\geq s\geq 2$ with $(t,s)\neq (3,2)$, and let $n$ be sufficiently large.

\vspace{1mm}
{\rm (i)}  \emph{(\cite{Zhu2020})}
$\mathrm{EX}(n,K_{2,2}^{odd})=\{H'(n,2,2)\}$;

\vspace{1mm}
{\rm (ii)}  \emph{(\cite{Peng-2024})}
$G_{3,3},G'_{3,3}\in \mathrm{EX}(n,K_{3,3}^{odd})$, and
$\mathrm{EX}(n,K_{s,t}^{odd})=\{G_{s,t}\}$ for $t\geq 4$.
\end{thm}

\begin{thm}\label{thm2.4}
Let $H^{odd}$ be an odd-ballooning of a graph $H$
in which each edge of $H$ is replaced by an
odd cycle of length at least five.
Then, $\mathcal{M}(H^{odd})$ is matching-good
if one of the following conditions holds:

\vspace{1mm}
{\rm (i)}  $H$ is a nontrivial tree, or

\vspace{1mm}
{\rm (ii)}  $H$ is isomorphic to $K_{s,t}$, where $t\geq s\geq 2$ and $(t,s)\neq (3,2)$.
\end{thm}

\subsubsection{Edge blow-up of graphs}

Given a graph $H$ and an integer $p\geq 2$, the edge blow-up of $H$, denoted by $H^{p+1}$,
is the graph obtained by replacing each edge in $H$ with a $K_{p+1}$ where the new vertices of $K_{p+1}$ are all distinct.
We define $\mathcal{H}(n,p,q,\mathcal{B})$ as the family of graphs of the form $Q\nabla T_{p}(n-q+1)$,
where $Q\in {\rm EX}(q-1,\mathcal{B})$.
In particular, $\mathcal{H}(n,p,1,\mathcal{B})=\{T_{p}(n)\}$.
The following result is due to Yuan.

\begin{thm}\label{thm1.11C}\emph{(\cite{Y2022})}
Let $H$ be a graph and $p\geq \chi(H)+1$.
For sufficiently large $n$, if $H$ is either non-bipartite,
or bipartite with $\gamma\big(\mathcal{M}(H^{p+1})\big)<\gamma(H)$,
then
$${\rm EX}(n,H^{p+1})
=\mathcal{H}\big(n,p,\gamma\big(\mathcal{M}(H^{p+1})\big),\mathcal{B}\big(H^{p+1}\big)\big).$$
\end{thm}

The Tur\'an problem for $H^{p+1}$ has been studied for several
classes of graphs $H$, including cycles~\cite{L2013}
and trees~\cite{Chi2023,L2013,Wang2021}.
For further related extremal results, see
\cite{M1968,Simonovits1966}.
Up to now, the spectral-consistency of $H^{p+1}$ has been established for several classes of graphs $H$,
including stars \cite{CFTZ2020,DKL2022},
matchings \cite{NWK2023}, and star forests \cite{WANG-2024}.
In this paper, we establish the following spectral-consistency result for the edge blow-up of graphs.

\begin{thm}\label{thm1.13C}
Let $H$ be a graph with at least one edge,
and let $p\geq \chi(H)+1$ be an integer.
Then \(\mathcal{M}(H^{p+1})\) is matching-good provided that \(H\) belongs to one of the following families:

\vspace{1mm}
{\rm (i)}  non-bipartite graphs;

\vspace{1mm}
{\rm (ii)} bipartite graphs satisfying $\gamma(\mathcal{M}(H^{p+1}))<\gamma(H)$;

\vspace{1mm}
{\rm (iii)} even cycles.
\end{thm}

Combining parts {\rm (i)} and {\rm (ii)} of
Theorem \ref{thm1.13C} with Theorem \ref{thm1.1},
we extend the corresponding spectral-consistency
result in \cite{Fang2026} to the full range
$p\geq\chi(H)+1$.

The remainder of this paper is organized as follows.
Section \ref{section2} collects the technical lemmas required for our analysis.
Section \ref{section3} is dedicated to the proof of our main result, Theorem \ref{thm1.1}.
We then show that the decomposition families of four classes of graphs are matching-good,
thereby proving Theorem \ref{thm1.9A} in Section \ref{section4},
and Theorems \ref{TH1.6}, \ref{thm2.4}, and \ref{thm1.13C} in Section \ref{section5}.
Finally, we conclude with several remarks and a discussion of related problems in the final section.

\section{Preliminaries}\label{section2}

In this section, we introduce several lemmas that will be used later to prove our theorems.
Let $\nu(G)$ be the maximum size of a matching in a graph $G$.
The following lemma is well-known.

\begin{lem}\label{LEMMA2.1}\emph{(\cite[p.~118]{Konig-1931})}
Let $G$ be a bipartite graph. Then, $\beta(G)=\nu(G)$.
\end{lem}

For a graph family $\mathcal{H}$ with $p(\mathcal{H})=p\geq 2$,
let $q(\mathcal{H})$ be the minimum integer $q$
such that $H\subseteq E_q\nabla T_{p}(p|H|)$ for some $H\in \mathcal{H}$.
The following elegant result was given by Zhang \cite{Zhang2025+}.

\begin{lem}\label{LEMMA2.2}\emph{(\cite{Zhang2025+})}
Let $\mathcal{H}$ be a finite family
with $p(\mathcal{H})=p\geq 2$, $q(\mathcal{H})=q$, and
$\max_{H\in \mathcal{H}}|H|=t$.
Assume that $H_0\subseteq M_t \nabla T_{p-1}(t(p-1))$ for some $H_0\in \mathcal{H}$.
Then, for sufficiently large $n$,
there exists a constant $\alpha$ such that every graph $G\in{\rm SPEX}(n,\mathcal{H})$
has a partition $V(G)=W\cup (\cup_{i\in [p]}V_i)$ satisfying the following property:

\vspace{1mm}
{\rm (i)}  $|W|=q-1$, and $|V_i|\in \{\lfloor\frac{n-q+1}{p}\rfloor, \lceil\frac{n-q+1}{p}\rceil\}$ for each $i\in [p]$;

\vspace{1mm}
{\rm (ii)}  For each $i\in [p]$, there exists a subset $U_i\subseteq V_i$
such that $|V_i\setminus U_i|\leq \alpha$ and $N_G(v)=V(G)\setminus V_i$ for all $v\in U_i$.
\end{lem}

Lemma \ref{LEMMA2.2} implies that there exists a bounded number of vertices in $G$ whose removal results in a complete $p$-partite graph.
For the convenience of proving Theorem \ref{thm1.1},
we now restate the above lemma from the perspective of the decomposition family as follows.

\begin{lem}\label{LEMMA2.3}
Let $\mathcal{H}$ be a finite graph family
such that $p(\mathcal{H})=p\geq 2$
and $\mathcal{M}(\mathcal{H})$ contains a matching.
Then, for sufficiently large $n$,
there exists a positive constant $\varphi=\varphi(\mathcal{H})$ such that every graph $G\in{\rm SPEX}(n,\mathcal{H})$
has a partition $V(G)=W\cup (\cup_{i\in [p]}V_i)$ satisfying the following conditions:

\vspace{1mm}
{\rm (i)} $|W|=\gamma-1$, and $|V_i|\in \{\lfloor\frac{n-\gamma+1}{p}\rfloor, \lceil\frac{n-\gamma+1}{p}\rceil\}$ for each $i\in [p]$,
where $\gamma=\gamma(\mathcal{M}(\mathcal{H}))$;

\vspace{1mm}
{\rm (ii)} For each $i\in [p]$, there exists a subset $V_i'\subseteq V_i$
such that $|V_i\setminus V_i'|= \varphi$ and $N_G(v)=V(G)\setminus V_i$ for all $v\in V_i'$.
\end{lem}

\begin{proof}
We first prove that $\gamma(\mathcal{M}(\mathcal{H}))=q(\mathcal{H})$.
By the definition of $\gamma(\mathcal{M}(\mathcal{H}))$,
there exists a member $H'\in \mathcal{H}$ and a bipartite member $M'\in \mathcal{M}(H')$ such that $\gamma(M')=\gamma$.
Thus, $M'\subseteq K_{\gamma,t}$,
where $t=\max_{H\in \mathcal{H}}|H|$.
Combining this with the definition of $\mathcal{M}(H')$,
we observe that $H'\subseteq K_{\gamma,t}\nabla T_{p-1}(t(p-1))=E_{\gamma}\nabla T_{p}(tp)$.
Hence, $\gamma(\mathcal{M}(\mathcal{H}))=\gamma\geq q(\mathcal{H})$.

Conversely, based on the definition of $q(\mathcal{H})$,
there exists a member $H\in \mathcal{H}$ such that $H\subseteq E_q\nabla T_{p}(tp)$.
By choosing the first two partite sets, we know that the decomposition family $\mathcal{M}(H)$ of $H$
contains a member $M$ such that $M\subseteq K_{q,t}$.
Therefore, $\gamma(\mathcal{M}(\mathcal{H}))\leq\gamma(M)\leq q=q(\mathcal{H}).$

Clearly, $\mathcal{M}(\mathcal{H})$ contains a matching if and only if
there exists some $H_0\in \mathcal{H}$ such that $H_0\subseteq M_t \nabla T_{p-1}(t(p-1))$.
Now, by Lemma \ref{LEMMA2.2}, statement (i) holds,
and for any $i\in [p]$, there exists $U_i\subseteq V_i$
such that $|V_i\setminus U_i|\leq \alpha$ and $N_G(v)=V(G)\setminus V_i$ for each $v\in U_i$.
Set $\varphi=\lceil\alpha\rceil+1$.
For each $i\in [p]$,
since $U_i\subseteq V_i$ and $|V_i\setminus U_i|\leq \alpha$,
there exists a subset $V_i'\subseteq U_i$ such that $|V_i\setminus V_i'|=\lceil\alpha\rceil+1=\varphi$.
Since $V_i'\subseteq U_i$,
we also have $N_G(v)=V(G)\setminus V_i$ for any $v\in V_i'$,
thus completing the proof.
\end{proof}

The following special family of graphs was introduced by Simonovits:

\begin{definition}\label{def2.1AB}\emph{(\cite{Simonovits1974})}
Denote by $\mathcal{D}(n,p,\varphi)$ the family of $n$-vertex graphs $G$ satisfying the
following symmetry condition:

\vspace{1mm}
{\rm (i)} It is possible to remove at most $\varphi$ vertices from $G$, resulting in a graph
$G'=\nabla_{i=1}^{p}G_i$,
where $\big||G_i|-\frac{n}{p}\big|\leq \varphi$ for each $i\in [p]$;

\vspace{1mm}
{\rm (ii)} For each $i\in [p]$, there exists a connected graph $F_i$ such that $G_i=k_iF_i$,
where $k_i=|G_i|/|F_i|$ and $|F_i|\leq \varphi$.
Moreover, any two copies $F_i',F_i''$ of $F_i$ in $G_i$ are symmetric subgraphs of $G$:
there exists an isomorphism $\phi: V(F_i')\to V(F_i'')$ such that for any $u\in V(F_i')$
and $v\in V(G)\setminus V(G')$, $uv\in E(G)$ if and only if $\phi(u)v\in E(G)$.
\end{definition}

To prove Theorem \ref{thm1.9A}, we rely on a result established by Simonovits \cite{Simonovits1974}.

\begin{lem}\emph{(\cite[Theorems~1 and~2]{Simonovits1974})}\label{LEMMA2.4}
Let $\mathcal{H}$ be a finite graph family
such that $p(\mathcal{H})=p\geq 2$.
If $\mathcal{M}(\mathcal{H})$ contains a linear forest,
then there exist $\varphi=\varphi(\mathcal{H})$ and $n_0=n_0(\varphi)$ such that
$\mathcal{D}(n,p,\varphi)$ contains a graph $G\in \mathrm{EX}(n,\mathcal{H})$ for $n\geq n_0$.
Furthermore, if $G$ is the only extremal graph in $\mathcal{D}(n,p,\varphi)$,
then it is the unique graph in $\mathrm{EX}(n,\mathcal{H})$ for every sufficiently large $n$.
\end{lem}

\section{Proof of Theorem \ref{thm1.1}}\label{section3}

Let $\mathcal{H}$ be a finite graph family
whose decomposition family $\mathcal{M}(\mathcal{H})$ is matching-good,
and let $G$ be any graph in $\mathrm{SPEX}(n,\mathcal{H})$.
Theorem \ref{thm1.1} asserts that $G$ must belong to $\mathrm{EX}(n,\mathcal{H})$.

From Definition \ref{definition-1.2}, we know that $p(\mathcal{H})=p\geq 2$,
and that $\mathcal{M}(\mathcal{H})$ contains a matching.
Hence, Lemma \ref{LEMMA2.3} is applicable.
Now, let $W,\varphi,\gamma,V_i$, and $V_i'$ (where $i\in [p]$) be defined as in Lemma \ref{LEMMA2.3}.
Set $V''_i=V_i\setminus V_i'$ for each $i\in [p]$.
By Lemma \ref{LEMMA2.3},
we have $|V_i''|=\varphi\geq 1$ and $|V_i|\in \{\lfloor\frac{n-\gamma+1}{p}\rfloor, \lceil\frac{n-\gamma+1}{p}\rceil\}$ for $i\in [p]$.

Now, set $Q=W\cup (\cup_{i\in [p]}V_i'')$ and
$\eta=|Q|$.
Clearly, $\eta=(\gamma-1)+p\varphi\geq p$, and thus $Q\neq \varnothing$.
Since $p(\mathcal{H})=\min_{H\in \mathcal{H}}\chi(H)-1\geq2$,
it follows that $T_p(n)$ is $\mathcal{H}$-free.
Using the Rayleigh quotient and the well-known inequality $e(T_p(n))\geq \frac{p-1}{2p}n^2-\frac{p}{8}$, we obtain
\begin{equation}\label{equation-1}
  \rho(G)
  \geq \rho(T_p(n))
   \geq \frac{\mathbf{1}^{\mathrm{T}}A(T_p(n))\mathbf{1}}{\mathbf{1}^{\mathrm{T}}\mathbf{1}}
  =\frac{2e(T_p(n))}{n}
  \geq \frac{p-1}{p}n-\frac{p}{4n}.
\end{equation}

Since $n$ is sufficiently large, $V_i'\neq\varnothing$
for every $i\in[p]$.
As $p\geq2$, we may choose vertices
$x\in V_1'$ and $y\in V_2'$.
By the definition of $V_i'$, we have $xy\in E(G)$,
and every vertex of $G$ is adjacent to
at least one of $x$ and $y$.
Therefore, $G$ is connected.
Hence, $G$ has a positive unit Perron vector
$\mathbf{x}=(x_1,\ldots,x_n)^{\mathrm T}$.
Choose $u^*\in Q$ such that $x_{u^*}=\max\{x_v~|~v\in Q\}$, and
$v^*\in V(G)\setminus Q$ such that  $x_{v^*}=\max\{x_v~|~v\in V(G)\setminus Q\}$.
Then
\begin{equation*}
\rho(G)x_{u^*}=\sum\limits_{u\in N_{G}(u^*)\cap Q}x_u+\sum\limits_{u\in N_{G}(u^*)\setminus Q}x_u
\leq \eta x_{u^*}+nx_{v^*}.
\end{equation*}
Combining this with \eqref{equation-1} and noting that $p\geq 2$, we obtain that for sufficiently large $n$,
\begin{equation}\label{equation-2}
x_{u^*}\leq \frac{n}{\rho(G)-\eta}x_{v^*}<2.1x_{v^*}.
\end{equation}
Assume that $v^*\in V'_{i^*}$ for some $i^*\in [p]$.
Then
\begin{align*}
 \rho(G)x_{v^*}
 \leq  \sum\limits_{v\in Q}x_v+\sum\limits_{i\in [p]\setminus\{i^*\}}\sum\limits_{v\in V'_i}x_v
 \leq  \eta x_{u^*}+\sum\limits_{i\in [p]\setminus\{i^*\}}\sum\limits_{v\in V'_i}x_v.
\end{align*}
In view of \eqref{equation-1}, we have $\rho(G)>\frac{2.1}{6} n$.
Combining these two inequalities with \eqref{equation-2} gives
\begin{align}\label{equation-3}
\sum\limits_{i\in [p]\setminus\{i^*\}}\sum\limits_{v\in V'_i}x_v
\geq \big(\rho(G)-2.1\eta \big)x_{v^*}>\big(1-\frac{6\eta}{n}\big)\rho(G)x_{v^*}.
\end{align}

In what follows, we present several claims.

\begin{claim}\label{lemma3.09ABC}
 For each $v\in V(G)$, we have $x_v\geq \big(1-\frac{8\eta}{n}\big)x_{v^*}$.
\end{claim}

\begin{proof}
Suppose, to the contrary, that there exists a vertex $u_0$ such that $x_{u_0}<(1-\frac{8\eta}{n})x_{v^*}$.
Let $G'$ be the graph obtained from $G$ by deleting all edges incident to $u_0$
and adding all possible edges between $u_0$ and $V'=(\bigcup_{i\in [p]\setminus \{i^*\}}V'_i)\setminus \{u_0\}$.

We aim to prove that $G'$ is $\mathcal{H}$-free.
Suppose, for the sake of contradiction, that $G'$ contains a subgraph $H\in\mathcal{H}$.
From the construction of $G'$,
we conclude that $u_0\in V(H)$ and $N_H(u_0)\subseteq V'$.

By Lemma \ref{LEMMA2.3}, we have $|V'_{i^*}|=\Theta(n/p)$.
Since $\mathcal{H}$ is finite, $|H|$ is bounded by a constant
independent of $n$. Thus, for $n$ sufficiently large, we have
$d_H(u_0)<|H|<|V_{i^*}'|$.
Choose a vertex $u\in V'_{i^*}\setminus V(H)$.
By Lemma \ref{LEMMA2.3}, $V'\subseteq N_G(u)$, and thus $N_H(u_0)\subseteq V'\subseteq N_G(u)$.
Therefore, the subgraph $G[(V(H)\setminus \{u_0\})\cup\{u\}]$ contains a copy of $H$,
leading to a contradiction.

On the other hand,
since $x_{u_0}<\big(1-\frac{8\eta}{n}\big)x_{v^*}<x_{v^*}$, by \eqref{equation-3}, we have
\begin{align*}
  \sum\limits_{v\in V'}\!\!x_v\!-\!\!\sum\limits_{v\in N_G(u_0)}\!\!x_{v}\geq\sum\limits_{i\in [p]\setminus \{i^*\}}\sum\limits_{v\in V'_i}x_v\!-\!x_{u_0}\!-\!\rho(G)x_{u_0}
  >\big(\frac{2\eta}{n}\rho(G)\!-\!1\big)x_{v^*}\geq 0.
\end{align*}
This leads to:
\begin{align}\label{0004}
  \rho(G')\!-\!\rho(G) \geq  \mathbf{x}^{\mathrm{T}}\big(A(G')\!-\!A(G)\big)\mathbf{x}
                  = 2x_{u_0}\Big(\sum\limits_{v\in V'}x_v-\!\!\!\sum\limits_{v\in N_G(u_0)}\!\!x_v\Big)>0,
\end{align}
which contradicts the fact that $G\in {\rm SPEX}(n,\mathcal{H})$.
The result follows.
\end{proof}

By Lemma \ref{LEMMA2.3}, $N_G(v)=V(G)\setminus V_i$ for each $v\in V_i'$,
and $N_G(u)\supseteq \cup_{j\in [p]\setminus\{i\}}V_j'$ for each $u\in V_i''$.
Hence, $x_v$ is constant for $v\in V_i'$,
and we may assume that $x_v=x_i$ for each $v\in V_i'$, where $i\in [p]$.
Furthermore,
for any $u\in V''_i$, it is not hard to verify that
\begin{align*}
-\eta x_{u^*}\leq -\sum\limits_{\widetilde{u}\in Q}x_{\widetilde{u}} \leq \rho(G)(x_u-x_{i})\leq \sum\limits_{\widetilde{u}\in V''_i}x_{\widetilde{u}}\leq \eta x_{u^*}.
\end{align*}
Recall that $\eta$ is constant.
Based on inequality \eqref{equation-2} and Claim \ref{lemma3.09ABC}, we obtain
\begin{align}\label{equation-4}
x_{u^*}<2.1x_{v^*}\leq \frac{2.1}{1-\frac{8\eta}{n}}x_i<3x_i
\end{align}
for each $i\in [p]$. Hence,
$-3\eta x_{i}\leq \rho(G)(x_u-x_{i})\leq3\eta x_{i}.$
By \eqref{equation-1}, we have $\rho(G)\geq\frac{3}{7}n$, and thus
\begin{align}\label{equation-5}
1-\frac{7\eta~}{n}\leq \frac{\rho(G)-3\eta~}{\rho(G)}\leq \frac{x_{u}}{x_i}
            \leq \frac{\rho(G)+3\eta~}{\rho(G)}\leq 1+\frac{7\eta~}{n}
\end{align}
for any $i\in [p]$ and any $u\in V_i''$. Clearly, $|V''_i|\leq \eta$ for each $i\in [p]$. It follows that
\begin{align*}
\big(|V''_i|\!-\!\frac{7\eta^2}{n}\big)x_i
\leq |V''_i|\big(1\!-\!\frac{7\eta}{n}\big)x_i
\leq \sum\limits_{\widetilde{u}\in V''_i}x_{\widetilde{u}}
\leq |V''_i|\big(1\!+\!\frac{7\eta}{n}\big)x_i
\leq \big(|V''_i|\!+\!\frac{7\eta^2}{n}\big)x_i.
\end{align*}
Then, there exists an $\epsilon_i\in [-\frac{7\eta^2}{n},\frac{7\eta^2}{n}]$ such that
$\sum_{\widetilde{u}\in V''_i}x_{\widetilde{u}}=(|V''_i|+\epsilon_i)x_i$.
Therefore,
$\sum_{v\in V_i}x_v=|V'_i|x_i+(|V''_i|+\epsilon_i)x_i=(|V_i|+\epsilon_i)x_i.$
Since $\rho(G)x_i=\sum_{v\in V(G)\setminus V_i}x_v$, we obtain
$(\rho(G)+|V_i|+\epsilon_i)x_i=\sum_{v\in V(G)}x_v$.
Set $n_i=|V_i|$ and $n_i^*=|V_i|+\epsilon_i$ for each $i\in [p]$.
Then, $n_i\geq\lfloor\frac{n-\gamma+1}{p}\rfloor$ and
\begin{align}\label{equation-6}
x_i=\frac{1}{\rho(G)+n_i^*}\sum\limits_{v\in V(G)}x_v.
\end{align}

\begin{claim}\label{claim-A3.1}
For every $u\in V(G)\setminus W$, we have $1\!-\!\frac{30\eta}{n}\!\leq\!\sqrt{n}x_u\!\leq\!1\!+\!\frac{30\eta}{n}$.
If $W\neq \varnothing$, then for each $w\in W$,
$\frac{p}{p-1}\!-\!\frac{120\eta}{n}\!\leq\!\sqrt{n}x_w\!\leq\!\frac{p}{p-1}\!+\!\frac{120\eta}{n}$.
\end{claim}

\begin{proof}
By Lemma \ref{LEMMA2.3}, for $i\in[p]$, we have
$|n_1\!-\!n_i|\leq 1$, and thus $|n^*_1\!-\!n^*_i|=|(n_1\!-\!n_i)\!+\!(\epsilon_1\!-\!\epsilon_i)|<1.1$.
Recall that $\eta=(\gamma-1)+p\varphi$.
From inequality \eqref{equation-1}, we obtain
\begin{align*}
\rho(G)\!+\!n_i^*\!\geq\!\frac{p-1}{p}n\!-\!\frac{p}{4n}\!
+\!\lfloor\frac{n-\gamma+1}{p}\rfloor\!+\!\epsilon_i\!\geq\!n\!-\!2\eta.
\end{align*}
Thus, by \eqref{equation-6}, we can deduce that
\begin{align}\label{0000}
\Big|\frac{x_i}{x_1}-1\Big|=\Big|\frac{\rho(G)+n_1^*}{\rho(G)+n_i^*}-1\Big|
\leq\frac{1.1}{\rho(G)+n_i^*}\leq\frac{1.2}{n}.
\end{align}

Given $u\in V(G)\setminus W$, we may assume that $u\in V_i$ for some $i\in [p]$.
Notice that $\frac{x_u}{x_1}=\frac{x_u}{x_i}\cdot\frac{x_i}{x_1}$.
Combining this with \eqref{equation-5} and \eqref{0000} yields
\begin{align}\label{0001}
1\!-\!\frac{9\eta}{n}\!\leq\!\big(1\!-\!\frac{1.2}{n}\big)\big(1\!-\!\frac{7\eta}{n}\big)\!\leq\!\frac{x_u}{x_1}
\!\leq\!\big(1\!+\!\frac{1.2}{n}\big)\big(1\!+\!\frac{7\eta}{n}\big)\!\leq\!1\!+\!\frac{9\eta}{n}.
\end{align}

Now, assume that $W\neq \varnothing$, and
choose an arbitrary vertex $w\in W$. From \eqref{0001}, we know that
\begin{align*}
\rho(G)x_w\geq\!\!\sum\limits_{v\in \cup_{i\in [p]}V'_i}x_v
\geq(n-\eta)\big(1-\frac{9\eta}{n}\big)x_1\geq (n-10\eta)x_1.
\end{align*}
In view of \eqref{equation-4}, we have $x_{u^*}<3x_{1}$.
Combining this with \eqref{0001} gives
\begin{align*}
\rho(G)x_w
=\!\!\sum\limits_{u\in N_{W}(w)}x_u+\!\!\sum\limits_{u\in N_{G}(w)\setminus W}x_u
\leq \eta\cdot 3x_{1}+n\big(1+\frac{9\eta}{n}\big)x_{1}
= (n+12\eta)x_1.
\end{align*}

On the other hand, by Lemma \ref{LEMMA2.3}, we have $\lfloor\frac{n-\gamma+1}{p}\rfloor\leq n_1\leq\lceil\frac{n-\gamma+1}{p}\rceil$.
Thus, by \eqref{0001}, we obtain that
\begin{align*}
\rho(G)x_1
\geq\!\!\!\sum\limits_{v\in \cup_{i=2}^pV_i}x_v
\geq\big(n\!-\!\eta\!-\!\big\lceil\frac{n-\gamma+1}{p}\big\rceil\big)\big(1\!-\!\frac{9\eta}{n}\big)x_1
\geq \big(\frac{p-1}{p}n\!-\!12\eta\big)x_1.
\end{align*}
Moreover, by \eqref{equation-4}, we know that $x_{u^*}<3x_1$. Thus, by \eqref{0001}, we obtain
\begin{align*}
\rho(G)x_1
&\leq\sum\limits_{w\in W}\!x_w+\!\!\!\sum\limits_{v\in \cup_{i=2}^pV_i}\!x_v
\leq \eta \cdot 3x_1+\big(n-\big\lfloor\frac{n-\gamma+1}{p}\big\rfloor\big)
\big(1+\frac{9\eta}{n}\big)x_1\\
&\leq \big(\frac{p-1}{p}n+14\eta\big)x_1.
\end{align*}
Note that $\frac12\leq \frac{p-1}{p}<1$ as $p\geq2$.
Combining the above four inequalities, we derive that
\begin{align}\label{0002}
\frac{p}{p\!-\!1}\!-\!\frac{76\eta~}{n}\!\leq\!\frac{n\!-\!10\eta~}{(1\!-\!1/p)n\!+\!14\eta~}
\!\leq\!\frac{x_w}{x_1}\!\leq\!\frac{n\!+\!12\eta~}{(1\!-\!1/p)n\!-\!12\eta~}\!
\leq\!\frac{p}{p\!-\!1}\!+\!\frac{76\eta~}{n}.
\end{align}

Now, in view of \eqref{0001} and \eqref{0002}, we obtain $\sum_{v\in V(G)}x^2_v\geq n\big(1\!-\!\frac{9\eta}{n}\big)^2x^2_1
\geq (n\!-\!18\eta)x^2_1$ and
\begin{equation*}
\sum\limits_{v\in V(G)}\!\!x^2_v
\leq \eta\big(\frac{p}{p-1}\!+\!\frac{76\eta}{n}\big)^2x^2_1+n\big(1\!+\!\frac{9\eta}{n}\big)^2x^2_1
\leq (n\!+\!24\eta)x^2_1.
\end{equation*}

Since $\sum_{v\in V(G)}x^2_v=\mathbf{x}^{\mathrm{T}}\mathbf{x}=1$, it follows that $\frac{1}{\sqrt{n+24\eta}} \leq x_1\leq \frac{1}{\sqrt{n-18\eta}}$,
which further implies that
$\big(1\!-\!\frac{20\eta~}{n}\big)\frac{1}{\sqrt{n}}\leq x_1\leq \big(1\!+\!\frac{20\eta}{n}\big)\frac{1}{\sqrt{n}}.$
Combining this with \eqref{0001},
for any $u\in V(G)\setminus W$, we have
\begin{align*}
x_u\leq \big(1+\frac{9\eta}{n}\big)x_1
\leq \big(1+\frac{9\eta}{n}\big)\big(1+\frac{20\eta}{n}\big)\frac{1}{\sqrt{n}}\leq \big(1+\frac{30\eta}{n}\big)\frac{1}{\sqrt{n}},
\end{align*}
and similarly,
$x_u\geq \big(1\!-\!\frac{9\eta}{n}\big)x_1
\geq \big(1\!-\!\frac{30\eta}{n}\big)\frac{1}{\sqrt{n}},$
as desired.

If $W\neq \varnothing$,
for any $w\in W$, by \eqref{0002}, we have
\begin{align*}
x_w\leq\big(\frac{p}{p-1}\!+\!\frac{76\eta}{n}\big)x_1
\leq \big(\frac{p}{p-1}\!+\!\frac{76\eta}{n}\big)\big(1\!+\!\frac{20\eta}{n}\big)\frac{1}{\sqrt{n}}\leq \big(\frac{p}{p-1}\!+\!\frac{120\eta}{n}\big)\frac{1}{\sqrt{n}},
\end{align*}
and similarly,
$x_w\geq\big(\frac{p}{p-1}\!-\!\frac{76\eta}{n}\big)x_1
\ge\big(\frac{p}{p-1}\!-\!\frac{120\eta}{n}\big)\frac{1}{\sqrt{n}}.$
This completes the proof.
\end{proof}

Recall that $V(G)=W\cup (\cup_{i\in [p]}V_i)$,
where $|W|=\gamma-1$ and $|V_i|\in \{\lfloor\frac{n-\gamma+1}{p}\rfloor, \lceil\frac{n-\gamma+1}{p}\rceil\}$ for $i\in [p]$.
Let $K= E_{\gamma-1}\nabla T_p(n-\gamma+1)$,
and assume that $V(E_{\gamma-1})=W$ and $V_1,\dots,V_p$ are the $p$ color classes of $T_p(n-\gamma+1)$.
We now define four edge sets as follows:
\begin{itemize}\setlength{\itemsep}{0pt}
\item $E_1^*=E\big(G[W]\big)$;
\item $E_2^*=E\big(K[W,\cup_{i\in [p]}V''_i]\big)\,\setminus\,E\big(G[W,\cup_{i\in [p]}V''_i]\big)$;
\item $E_3^*=E\big(K[\cup_{i\in [p]}V''_i]\big)\,\setminus\,E\big(G[\cup_{i\in [p]}V''_i]\big)$;
\item $E_4^*=E\big(G[\cup_{i\in [p]}V''_i]\big)\,\setminus\,E\big(K[\cup_{i\in [p]}V''_i]\big)$.
\end{itemize}
Then, it is clear that $E(G)\setminus(E_1^*\cup E_4^*)=E(K)\setminus(E_2^*\cup E_3^*)$.

\begin{claim}\label{claim-A3.B}
$|E_1^*|\leq {\rm ex}\big(\gamma-1,\mathcal{B}(\mathcal{H})\big)$.
\end{claim}

\begin{proof}
Note that $|W|=\gamma-1$.
If $W=\varnothing$, then $E_1^*=\varnothing$, and we are done.
If $W\neq\varnothing$, it suffices to show that $G[W]$ is $\mathcal{B}(\mathcal{H})$-free.

By the definition of $\mathcal{B}(\mathcal{H})$,
if $\beta=\gamma$, then $\mathcal{B}(\mathcal{H})=\{K_{\gamma}\}$.
Furthermore, since $|W|=\gamma-1$, $G[W]$ is clearly $\mathcal{B}(\mathcal{H})$-free.
If $\beta<\gamma$,
then
$$\mathcal{B}(\mathcal{H})=\{M[S]~|~M\in \mathcal{M}(\mathcal{H}),~S~\text{is a covering of}~M~\text{with}~|S|<\gamma\}.$$
Suppose, to the contrary, that $G[W]$ contains a subgraph isomorphic to some member in $\mathcal{B}(\mathcal{H})$.
Then,
there exist an $M\in \mathcal{M}(\mathcal{H})$ and a covering $S$ of $M$ with
$|S|<\gamma$ such that $M[S]\subseteq G[W]$.
Furthermore, since $G[W,V'_1]$ is complete bipartite, we have $M\subseteq G[W\cup V'_1]$.
By Definition \ref{definition-1.1},
$G[W\cup (\cup_{i\in [p]}V'_i)]$ contains a member of $\mathcal{H}$, a contradiction.
Therefore, $G[W]$ is $\mathcal{B}(\mathcal{H})$-free.
\end{proof}

Now, we complete the proof of Theorem \ref{thm1.1} by utilizing the structures of $K$, $G$,
and an edge-extremal graph $G^{\star}$ with the improved properties described in Definition \ref{definition-1.2}.

\begin{proof}[\textbf{Proof of Theorem \ref{thm1.1}}]
By Definition \ref{definition-1.2},
 there exists a graph $G^{\star}\in {\rm EX}(n,\mathcal{H})$ such that
$G^{\star}=H_1\nabla H_2$,
where $H_1\in {\rm EX}(\gamma-1,\mathcal{B}(\mathcal{H}))$ and $H_2$ differs from $T_{p}(n-\gamma+1)$ by $O(1)$ edges.
Assume that $V(H_1)=W^\star$ and $V_1^\star,\dots,V_p^\star$ are the $p$ color classes of $T_p(n-\gamma+1)$,
i.e., $V(H_2)=\cup_{i\in [p]}V_i^{\star}$.
We now define three edge sets as follows:
\begin{itemize}\setlength{\itemsep}{0pt}
\item $E_1^{\star}=E\big(H_1\big)$;
\item $E_3^{\star}=E\big(T_{p}(n-\gamma+1)\big)\,\setminus\,E\big(H_2\big)$;
\item $E_4^{\star}=E\big(H_2\big)\,\setminus\,E\big(T_{p}(n-\gamma+1)\big)$.
\end{itemize}

Observe that $|W|=|W^\star|=\gamma\!-\!1$
and $\max_{i,j\in[p]}(|V_i|\!-\!|V_j|)=\max_{i,j\in[p]}(|V_i^\star|\!-\!|V_j^\star|)\leq1$.
We can assume that $W^{\star}=W$ and $V_i^\star=V_i$ for $i\in [p]$.
Recall that $E(G)\setminus(E_1^*\cup E_4^*)=E(K)\setminus(E_2^*\cup E_3^*)$.
Similarly, we observe that
$E(G^{\star})\setminus(E_1^{\star}\cup E_4^{\star})=E(K)\setminus E_3^{\star}$.
Consequently,
 \begin{align}\label{equation-8AB}
e(G)-e(G^\star)
=\big(|E_1^*|\!+\!|E_4^*|\!-\!|E_2^*|\!-\!|E_3^*|\big)\!
-\!\big(|E^{\star}_1|\!+\!|E^{\star}_4|\!-\!|E^{\star}_3|\big).
\end{align}

Next, we prove that $e(G)=e(G^\star)$.
Suppose, to the contrary, that $e(G)<e(G^\star)$.
Then,
\begin{align}\label{equation-8A}
|E_1^*|+|E_4^*|-|E_2^*|-|E_3^*|< |E^{\star}_1|+|E^{\star}_4|-|E^{\star}_3|.
\end{align}
Notice that $|E^{\star}_1|=e(H_1)={\rm ex}\big(\gamma-1,\mathcal{B}(\mathcal{H})\big)$.
By Claim \ref{claim-A3.B}, we have $|E_1^*|\leq |E^{\star}_1|$.
Since $e(G)<e(G^\star)$ and both edge counts are integers,
we have $e(G^\star)-e(G)\geq 1$.
Together with \eqref{equation-8AB} and
$|E_1^*|\leq |E_1^\star|$, this yields
\begin{align}\label{equation-7}
&\frac{p^2}{(p-1)^2}|E_1^\star|
  +|E_4^\star|-|E_3^\star|
  -\left(
    \frac{p^2}{(p-1)^2}|E_1^*|
    +|E_4^*|-\frac{p}{p-1}|E_2^*|-|E_3^*|
   \right)
   \nonumber\\
&\qquad =
  e(G^\star)-e(G)+\left(\frac{p^2}{(p-1)^2}-1\right) \left(|E_1^\star|-|E_1^*|\right)
  +\frac{1}{p-1}|E_2^*|
 \geq 1.
\end{align}

Recall that $|W\cup (\cup_{i\in [p]}V''_i)|=\eta=(\gamma-1)+p\varphi$.
It follows that $|E_i^*|\leq \binom{\eta}{2}$ for $i\in [4]$.
On the other hand,
$|E_1^{\star}|\leq e(K_{|W|})=\binom{\gamma-1}{2}$,
and since $H_2$ differs from $T_{p}(n-\gamma+1)$ by $O(1)$ edges,
there exists a constant $\varphi^{\star}$ such that
$|E_j^{\star}|\leq \varphi^{\star}$ for $j\in \{3,4\}$.

Arguing as in \eqref{0004} and using \eqref{equation-8AB} and Claim \ref{claim-A3.1},
we have
\begin{align*}
    &~~~~~~\rho(G^\star)-\rho(G)\\
    &\geq 2\Big(\!\sum\limits_{ww'\in E_1^\star}\!\!x_{w}x_{w'}\!+\!\!\!\sum\limits_{uv\in E_4^\star}\!\!x_{u}x_{v}\!-\!\!\!\sum\limits_{uv\in E_3^\star}\!\!x_{u}x_{v}\Big)
    \!-\!2\Big(\!\sum\limits_{ww'\in E_1^*}\!\!x_{w}x_{w'}\!+\!\!\!\sum\limits_{uv\in E_4^*}\!\!x_{u}x_{v}\!-\!\!\!\sum\limits_{wu\in E_2^*}\!\!x_{w}x_{u}\!-\!\!\!\sum\limits_{uv\in E_3^*}\!\!x_{u}x_{v}\Big) \\
  &\geq 2\big(\frac{p^2}{(p-1)^2}|E^{\star}_1|+|E^{\star}_4|-|E^{\star}_3|\big)\frac{1}{n}- 2\big(\frac{p^2}{(p-1)^2}|E_1^*|+|E_4^*|-\frac{p}{p-1}|E_2^*|-|E_3^*|\big)\frac{1}{n}
  -\frac{\sigma}{n^{2}}
\end{align*}
for some positive constant $\sigma$.
Combining this with \eqref{equation-7}, we obtain $\rho(G^{\star})>\rho(G)$,
which contradicts the fact that $G\in {\rm SPEX}(n,\mathcal{H})$.
Therefore, $e(G)=e(G^{\star})$, which implies that $G\in {\rm EX}(n,\mathcal{H})$.
This completes the proof.
\end{proof}

\section{Proof of Theorem \ref{thm1.9A}}\label{section4}

For each $i\in [t]$,
let $H_i$ be a finite graph with $\mathrm{ex}(n,H_i)=e(T_p(n))+O(1)$, where $n$ is sufficiently large.
Theorem \ref{thm1.9A} asserts that $\mathcal{M}(\cup_{i\in [t]}H_i)$ is matching-good.
By the Erd\H{o}s--Stone--Simonovits theorem and the assumption
$\mathrm{ex}(n,H_i)=e(T_p(n))+O(1)$,
we have $\chi(H_i)=p+1$ for every $i\in[t]$, and hence $p(H_i)=p$. Therefore, for
$H=\bigcup_{i\in[t]}H_i$,
we have $\chi(H)=p+1$ and thus $p(H)=p$.

By Theorem \ref{THM1.1A}, every $\mathcal{M}(H_i)$ contains a matching. Consequently,
$\mathcal{M}(\cup_{i\in [t]}H_i)$ also contains a matching.
By Lemma \ref{LEMMA2.4},
there exists a constant $\varphi=\varphi(\cup_{i\in [t]}H_i)$ such that
$\mathcal{D}(n,p,\varphi)$ contains a graph $G\in \mathrm{EX}(n,\cup_{i\in [t]}H_i)$.
By Definition \ref{def2.1AB}, $G$ satisfies the
following properties:

\begin{enumerate}
\item[($a$)] There exists a subset $R\subseteq V(G)$ with $|R|\leq \varphi$
such that $G-R=\nabla_{i=1}^{p}G_i$,
where $\big||G_i|-\frac{n}{p}\big|\leq \varphi$ for each $i\in [p]$;

\item [($b$)] For each $i\in [p]$, there exists a connected graph $F_i$ such that $G_i=k_iF_i$,
where $k_i=|G_i|/|F_i|$ and $|F_i|\leq \varphi$.
Moreover, any two copies $F_i',F_i''$ of $F_i$ in $G_i$ are symmetric subgraphs of $G$.
\end{enumerate}

Set $\phi=\max\big\{|H_j|: j\in[t]\}$ and $U_i=V(G_i)$.
By ($a$), for every $i\in [p]$, we have
\begin{align}\label{eq4.1}
t\phi\varphi\leq \frac{n}{p}-\varphi\leq |U_i|\leq \frac{n}{p}+\varphi.
\end{align}

We first prove that $U_i$ is an independent set for each $i\in [p]$.
If $U_i$ is not an independent set for some $i$, then $F_i$ must be nontrivial.
Recall that $\mathcal{M}(\cup_{i\in [t]}H_i)$ contains a matching, denoted by $M_{a}$,
where $a\leq \frac12\sum_{i\in [t]}|H_i|\leq\frac12t\phi.$
Since $|F_i|\leq\varphi$ for $i\in [p]$,
by \eqref{eq4.1}, we have $k_i={|G_i|}/{|F_i|}\geq a$.
Thus, $G_i$ contains a matching $M_{a}$.
By the definition of a decomposition family,
$G[\cup_{i\in [p]}U_i]$ must contain a copy of $\cup_{i=1}^{t}H_i$, a contradiction.
Therefore, every $U_i$ is an independent set.

By ($b$), any two vertices in $U_i$ are symmetric in $G$, i.e.,
$N_G(u_1)=N_G(u_2)$ for all $u_1,u_2\in U_i$.
This implies that for each $u\in R$ and each $i\in [p]$,
either $N_G(u)\supseteq U_i$ or $N_G(u)\cap U_i=\varnothing$.

Let $W=\big\{u\in R~|~N_G(u)\supseteq \cup_{i\in [p]}U_i\big\}$,
and for $i\in [p]$, let
$U'_i=\big\{u\in R~|~N_G(u)\cap \big(\cup_{j\in [p]}U_j\big)
=\cup_{j\in [p]\setminus \{i\}}U_j\big\}$.
Then, $e(G[U_i,U_i'])=0$ for $i\in [p]$,
and we have some claims.
The first one implies that
for each $u\in R$, there exists at most one $i\in [p]$ such that
$N_G(u)\cap U_i=\varnothing$.

\begin{claim}\label{CLAIM2.1}
$R=W\cup (\bigcup_{i\in [p]}U_i')$.
\end{claim}

\begin{proof}
Since each of $W$ and $U_i'$ is a subset of $R$,
it follows that $W\cup (\cup_{i\in [p]}U_i')\subseteq R$.
Now, we show $R\subseteq W\cup(\cup_{i\in [p]}U_i')$. Suppose, to the contrary,
that there exists $u_0\in R\setminus (W\cup (\cup_{i\in [p]}U_i'))$.
By the definitions of $W$ and $U_i'$, there exist two distinct $i_1,i_2\in [p]$
such that $N_G(u_0)\cap(U_{i_1}\cup U_{i_2})=\varnothing$.
Let $G'$ be the graph obtained from $G$ by deleting all edges incident to $u_0$
and adding all possible edges between $u_0$ and $\bigcup_{i\in [p-1]}U_i$.
In view of \eqref{eq4.1}, we have
$$d_G(u_0)\leq n-|U_{i_1}|-|U_{i_2}|<|\cup_{i\in [p-1]}U_i|,$$
and thus $e(G')>e(G)$.
By the choice of $G$, we know that
$G'$ is not $\cup_{i\in[t]}H_i$-free,
i.e., $G'$ contains a subgraph $H$ isomorphic to $\cup_{i\in [t]}H_i$.
From the construction of $G'$,
we have $u_0\in V(H)$.
Notice that $N_{H}(u_0)\subseteq \cup_{i\in [p-1]}U_i$ and $|V(H)|\leq t\phi<|U_p|$.
Select a vertex $u\in  U_p\setminus V(H)$.
By condition ($a$), we have
$N_H(u_0)\subseteq\bigcup_{i\in[p-1]}U_i\subseteq N_G(u)$.
This implies that $G[(V(H)\setminus \{u_0\})\cup\{u\}]$ contains a copy of $H$,
which leads to a contradiction.
Therefore, $R=W\cup (\cup_{i\in [p]}U_i')$.
\end{proof}

\begin{claim}\label{CLAIM2.2}
Let $|U_i\cup U_i'|=n_i$ for $i\in [p]$, and
assume $n_1\geq n_2\geq \dots \geq n_p$.
Then, $n_1-n_p\leq 1$.
\end{claim}

\begin{proof}
Suppose, for a contradiction, that $n_1 - n_p \geq 2$.
Choose a vertex $u_0 \in U_1$.
Let $G''$ be the graph obtained from $G$ by deleting all edges from $u_0$ to $U_p \cup U_p'$,
and adding all edges from $u_0$ to $(U_1 \cup U_1') \setminus \{u_0\}$.
Since $n_1 \geq n_p + 2$, it follows that $|(U_1 \cup U_1') \setminus \{u_0\}| > |U_p \cup U_p'|$, and hence $e(G'') > e(G)$.
This further implies that $G''$ contains a subgraph $H$ isomorphic to $\cup_{i\in [t]}H_i$.
By an argument similar to that in the proof of Claim \ref{CLAIM2.1},
we can deduce that $G$ contains a copy of $H$, a contradiction.
Therefore, $n_1 - n_p \leq 1$.
\end{proof}

\begin{claim}\label{claim-4.3B}
Let $H=\bigcup_{i\in[t]}H_i$. Then
$\beta(\mathcal{M}(H))=\gamma(\mathcal{M}(H))=t$.
\end{claim}

\begin{proof}
We first prove that $\beta(\mathcal{M}(H))\ge t$.
Let $M\in\mathcal{M}(H)$. By the definition of the decomposition
family, for some sufficiently large constant $s$,
we have $H\subseteq (M\cup E_s)\nabla T_{p-1}((p-1)s)$.
Fix such an embedding. For each $i\in[t]$, restrict this embedding
to $H_i$, and let $N_i$ be the subgraph of $M$ consisting of the
edges of $M$ used by the copy of $H_i$.

Since $\chi(H_i)=p+1$, the graph $N_i$ is nonempty; otherwise the
copy of $H_i$ would be contained in the $p$-partite graph
$E_s\nabla T_{p-1}((p-1)s)$. Moreover, by the minimality in the
definition of the decomposition family, $N_i$ contains some
$M_i\in\mathcal{M}(H_i)$.
Since $H_1,\ldots,H_t$ are vertex-disjoint, the graphs
$M_1,\ldots,M_t$ are pairwise vertex-disjoint subgraphs of $M$.
Consequently,
$\beta(M)\ge \sum_{i=1}^t\beta(M_i)\ge t$.
Since $M\in\mathcal{M}(H)$ was arbitrary, we obtain
$\beta(\mathcal{M}(H))\ge t$.

On the other hand, by Theorem~\ref{THM1.1A}, for every $i\in[t]$
there exists a star $S_i\in\mathcal{M}(H_i)$. Let
$S=\bigcup_{i=1}^t S_i$.
By taking the corresponding embeddings of $H_i$ vertex-disjointly,
we have
$H\subseteq (S\cup E_{s'})\nabla T_{p-1}((p-1)s')$
for some constant $s'$. Hence, by the definition of the
decomposition family, $S$ contains a member
$M^\ast\in\mathcal{M}(H)$.

Let $C$ consist of the centers of $S_1,\ldots,S_t$. Then $C$ is an
independent covering of $S$. Since $M^\ast\subseteq S$, the graph
$M^\ast$ is bipartite and $C\cap V(M^\ast)$ is an independent
covering of $M^\ast$. Hence,
$$\gamma(\mathcal{M}(H))\le \gamma(M^\ast)\le |C\cap V(M^\ast)|\le |C|=t.$$
Using $\beta(\mathcal{M}(H))\le\gamma(\mathcal{M}(H))$,
we conclude that
$$t\le\beta(\mathcal{M}(H))\le\gamma(\mathcal{M}(H))\le t.$$
Therefore,
$\beta(\mathcal{M}(H))=\gamma(\mathcal{M}(H))=t$.
\end{proof}

\begin{claim}\label{CLAIM2.3}
$|W|=t-1$.
\end{claim}

\begin{proof}
By Claim~\ref{claim-4.3B}, we have
$\gamma(\mathcal{M}(H))=t$, where
$H=\bigcup_{i\in[t]}H_i$. Hence, by the definition of $\gamma$,
there exists a bipartite graph
$M\in\mathcal{M}(H)$ such that $\gamma(M)=t$.

We first prove that $|W|\leq t-1$.
Suppose to the contrary that $|W|\geq t$.
Let $S$ be a minimum independent covering of $M$. Then $|S|=t$.
By the observation following Definition~\ref{definition-1.1}, we have
$|V(M)|\leq |V(H)|=\sum_{i\in[t]}|V(H_i)|\leq t\phi$.
Since $S$ is an independent covering, $V(M)\setminus S$ is also
independent, and every edge of $M$ has one endpoint in $S$ and the
other in $V(M)\setminus S$. Therefore,
$M\subseteq K_{|S|,\,|V(M)\setminus S|}\subseteq K_{t,t\phi}$.
Since $|W|\geq t$ and $|U_1|=\Theta(n)$, for sufficiently large
$n$ we have $|U_1|\geq t\phi$, and hence
$M\subseteq K_{t,t\phi}\subseteq K_{|W|,|U_1|}$.
By the definition of
$\mathcal{M}(\bigcup_{i\in[t]}H_i)$, it follows that
$\cup_{i \in [t]} H_i\subseteq K_{|W|,|U_1|,|U_2|,\dots,|U_p|}\subseteq G$,
a contradiction.
Hence $|W|\leq t-1$.

Now, we show that $|W|=t-1$.
Suppose, to the contrary, that $|W| \leq t - 2$.
Choose an arbitrary vertex $u_0 \in U_1$.
Let $G'$ be the graph obtained from $G$ by deleting all edges of $G[R]$ and all edges from $u_0$ to $W$,
and then adding all edges between $u_0$ and $U_1 \setminus \{u_0\}$.
Recall that $|R| \leq \varphi$, and by \eqref{eq4.1}, we have $|U_1|=\Theta(n)$.
Therefore, $|U_1 \setminus \{u_0\}|> e(G[R])+|W|$, and hence $e(G') > e(G)$.
Since $G\in \mathrm{EX}(n,\cup_{i\in [t]}H_i)$,
we conclude that
$G'$ contains a copy of $\cup_{i\in [t]}H_i$.

Recall that $e(G[U_i])=e(G[U_i,U_i'])=0$ for $i\in [p]$.
By the construction of $G'$, we can see that
$W\cup \{u_0\}$, $(U_1\setminus \{u_0\})\cup U_1'$,
and $U_i\cup U_i'$ for $i\in [p]\setminus\{1\}$ are all independent sets in $G'$.
Since $|U_i\cup U_i'|=n_i$ for $i\in [p]$,
it follows that $G'$ is a subgraph of $K_{|W|+1,n_1-1,n_2,\dots,n_p}$.
Furthermore, since $G'$ contains a copy of $\cup_{i\in [t]}H_i$,
so does $K_{|W|+1,n_1-1,n_2,\dots,n_p}$.
By Definition \ref{definition-1.1},
$K_{|W|\!+\!1,n_1\!-\!1}$ contains some bipartite member $M'$ of $\mathcal{M}(\cup_{i\in [t]}H_i)$.
Then, $\gamma(M')\leq \gamma(K_{|W|+1,n_1-1})=|W|+1\leq t-1$.

On the other hand, recall that $\gamma(\mathcal{M}(\cup_{i\in [t]}H_i))=t$, and thus
$\gamma(M')\geq\gamma(\mathcal{M}(\cup_{i\in [t]}H_i))=t$,
which leads to a contradiction.
Therefore, $|W|=t-1$.
\end{proof}

Let $H'=G[W]$ and $H''=G-W$.
Then, $V(H'')=\cup_{i\in [p]}(U_i\cup U_i')$.
By Claims \ref{CLAIM2.2} and \ref{CLAIM2.3},
we have $|H''|=n-t+1$, and for each $i\in [p]$,
$|U_i\cup U_i'|\in\{\lfloor\frac{n-t+1}{p}\rfloor,\lceil\frac{n-t+1}{p}\rceil\}$.
Combining the definition of $U_i'$ with ($a$) and ($b$),
we can conclude that $H''$ is obtained from $T_p(n-t+1)$ by adding and deleting $O(1)$ edges within $G[\cup_{i\in [p]}U_i']$.
If $t=1$, then $G=H''$. Note that $G\in \mathrm{EX}(n,\cup_{i\in [t]}H_i)$.
By Definition \ref{definition-1.2},
$\mathcal{M}(\cup_{i\in [t]}H_i)$ is clearly matching-good.
Next, let $t\geq 2$ and $W\!=\!\{w_1,\ldots,w_{t-1}\}$.

\begin{claim}\label{CLAIM2.4}
$H''$ is $\{H_i~|~i\in [t]\}$-free.
\end{claim}

\begin{proof}
Suppose to the contrary that $H''$ contains a subgraph
isomorphic to $H_{\ell}$ for some $\ell\in[t]$.
We may assume without loss of generality that $\ell=t$.
Since $|H_t|\leq \phi$, by \eqref{eq4.1} we have
\[
|U_i\setminus V(H_t)|
\geq |U_i|-|H_t|
\geq t\phi\varphi-\phi
\geq (t-1)\phi
\]
for every $i\in[p]$.
Hence, we may choose a subset
$V_i'=\{u_{i,1},\ldots,u_{i,(t-1)\phi}\}\subseteq U_i\setminus V(H_t)$.
We can partition $V_i'$ into the sets $V_{i,j}'$ for $j\in[t-1]$,
where $V_{i,j}'=\{u_{i,(j-1)\phi+1},u_{i,(j-1)\phi+2},\dots,u_{i,j\phi}\}$.
For each $j\in [t-1]$,
let $G_j$ be the subgraph of $G$ induced by
$(\cup_{i\in [p]}V_{i,j}')\cup \{w_j\}$.
Then, $G_j\cong K_1\nabla T_{p}(p\phi)$, and the subgraphs
$G_1,G_2,\ldots,G_{t-1}$ and $H_t$ are vertex-disjoint.
For each $j\in[t-1]$, recall that
$\mathcal{M}(H_j)$ contains a star, say
$K_{1,r_j}$ for some positive integer $r_j$.
By the definition of the decomposition family,
there exists a positive integer $s_j$ such that
\[
H_j\subseteq
(K_{1,r_j}\cup E_{s_j})\nabla
T_{p-1}((p-1)s_j).
\]
Let $c_j$ denote the center of $K_{1,r_j}$.
Since $\chi(H_j)=p+1$, the copy of $H_j$
in the above inclusion must contain $c_j$.
Moreover, deleting $c_j$ from the graph on
the right-hand side yields a complete
$p$-partite graph.
Thus, the remaining vertices of this copy
of $H_j$ can be partitioned into $p$
independent sets.
Since $|V(H_j)|\leq\phi$, each of these sets
contains at most $\phi$ vertices.
Consequently,
$H_j\subseteq K_1\nabla T_p(p\phi)\cong G_j$.
Consequently, $G$ contains the disjoint union of $H_1,H_2,\ldots,H_t$,
which leads to a contradiction.
Thus, the claim follows.
\end{proof}

\begin{claim}\label{CLAIM2.5}
$K_{t-1}\nabla H''$ is $\cup_{i\in [t]}H_i$-free.
\end{claim}

\begin{proof}
Suppose, to the contrary, that $K_{t-1}\nabla H''$ contains $\cup_{i\in [t]}H_i$ as a subgraph.
By Claim \ref{CLAIM2.4},
$H''$ is $\{H_i~|~i\in [t]\}$-free.
Then, each $H_i$ must contain at least one vertex from $K_{t-1}$.
This implies that $\cup_{i\in [t]}H_i$ contains at least $t$ vertices from $K_{t-1}$,
which leads to a contradiction.
\end{proof}

Now, we complete the proof of Theorem \ref{thm1.9A}.

\begin{proof}[\textbf{Proof of Theorem \ref{thm1.9A}}]

Observe that $G$ is a subgraph of $K_{t-1}\nabla H''$.
Since $G\in \mathrm{EX}(n,\cup_{i\in [t]}H_i)$,
by Claim \ref{CLAIM2.5}, we conclude that $G\cong K_{t-1}\nabla H''$.

By Claim \ref{claim-4.3B}, we have
$\beta(\mathcal{M}(\cup_{i\in [t]}H_i))=\gamma(\mathcal{M}(\cup_{i\in [t]}H_i))=t$.
In this case, by the definition of $\mathcal{B}(\cup_{i\in [t]}H_i)$,
we know that $\mathcal{B}(\cup_{i\in [t]}H_i)=\{K_{t}\}$, and hence, ${\rm EX}(t-1,\mathcal{B}(\cup_{i\in [t]}H_i))=\{K_{t-1}\}$.

Recall that $\mathcal{M}(\cup_{i\in [t]}H_i)$ contains a matching,
and $H''$ differs from $T_p(n-t+1)$ by $O(1)$ edges.
By Definition \ref{definition-1.2},
$\mathcal{M}(\cup_{i\in [t]}H_i)$ is clearly matching-good, thus completing the proof.
\end{proof}

\section{Proofs of Theorems \ref{TH1.6}, \ref{thm2.4}, and \ref{thm1.13C}}\label{section5}

In this section, we investigate three classes of graphs:
generalized color-critical graphs, odd-ballooning of trees or of complete bipartite graphs,
and edge blow-up of non-bipartite graphs or of specific bipartite graphs.
We prove that the decomposition families of these graphs are matching-good.

\begin{proof}[\textbf {Proof of Theorem \ref{TH1.6}}]
Let $\mathcal{H}$ be a $q$-color-critical graph family
with $p(\mathcal{H})=p\geq 2$.
Recall that $\beta(\mathcal{M}(\mathcal{H}))$ and $\gamma(\mathcal{M}(\mathcal{H}))$ represent
the minimum covering number and the minimum independent covering number of the graph family $\mathcal{M}(\mathcal{H})$, respectively.
Let $\beta(\mathcal{M}(\mathcal{H}))=\beta$ and $\gamma(\mathcal{M}(\mathcal{H}))=\gamma$.

By the definition of $\beta(\mathcal{M}(\mathcal{H}))$,
there exists a member $M\in \mathcal{M}(\mathcal{H})$ such that $\beta(M)=\beta$.
Thus, we can find a covering $S$ of $M$ such that $|S|=\beta$.
We now prove that $\beta\geq q$.
Suppose, for the sake of contradiction, that $\beta\leq q-1$.
Let $G^*$ be the graph obtained from $T_{p}(n)$ by embedding $M$ into one color class of $T_{p}(n)$.
Clearly, $\chi(G^*-S)=p$,
and by Definition \ref{definition-1.1}, $G^*$ must contain some member $H$ in $\mathcal{H}$.
Let $S_H=S\cap V(H)$.
Since $|S_H|\leq |S|=\beta\leq q-1$ and
$|V(H)|\geq q-1$, we may choose a set
$S'\subseteq V(H)$ such that
$S_H\subseteq S'$ and $|S'|=q-1$.
Then $H-S'\subseteq H-S_H\subseteq G^*-S$, and hence
$\chi(H-S')\leq \chi(G^*-S)\leq p$,
contradicting condition (i) of Definition \ref{definition-1.3}.
Therefore, $\beta\geq q$.

By the definition of $\gamma(\mathcal{M}(\mathcal{H}))$,
there also exists a bipartite member $M'\in \mathcal{M}(\mathcal{H})$ such that $\gamma(M')=\gamma$.
Consequently, $\gamma=\gamma(M')\geq \beta(M')\geq \beta.$
By Lemma \ref{LEMMA2.1},
$M'$ contains a matching $M_{\beta}$.
On the other hand, by condition (ii) of Definition \ref{definition-1.3}, we can observe that
$\mathcal{M}(\mathcal{H})$ must contain a matching $M_{q'}$ for some $q'\leq q$.
Since $q\leq\beta$,
it follows that $M_{q'}\subseteq M_{\beta}\subseteq M'$.
By Definition \ref{definition-1.1},
the members in $\mathcal{M}(\mathcal{H})$ are all minimal.
Since both $M_{q'}$ and $M'$ belong to $\mathcal{M}(\mathcal{H})$,
we conclude that $M_{q'}=M'$, which implies $\beta=q=q'$.

Since $M_q\in \mathcal{M}(\mathcal{H})$, we have $\gamma\leq \gamma(M_q)=q$.
This, together with $\gamma\geq \beta$ and $\beta=q$,
gives $\gamma=\beta=q$.
In this case, we know that $\mathcal{B}(\mathcal{H})=\{K_{q}\}$ and ${\rm EX}(q-1,\mathcal{B}(\mathcal{H}))=\{K_{q-1}\}$.
Moreover, by Theorem \ref{TH1.3}, $\mathrm{EX}(n,\mathcal{H})=\{H(n,p,q)\}$,
where $H(n,p,q)=K_{q-1}\nabla T_{p}(n-q+1)$.
Therefore, by Definition \ref{definition-1.2},
it is easy to verify that $\mathcal{M}(\mathcal{H})$ is matching-good.
\end{proof}

Recall that the odd-ballooning of $H$, denoted by $H^{odd}$, is a graph obtained from
$H$ by replacing each edge of $H$ with an odd cycle of length at least five,
where the new vertices of different odd cycles are distinct.
For any bipartite graph $H$ with at least one edge,
$H^{odd}$ is $3$-chromatic, and hence $p(H^{odd})=2$.

\begin{proof}[\textbf{Proof of Theorem \ref{thm2.4}}]
(i) Let $A$ and $B$ be the partite sets of $T$ with $|A|\leq |B|$, and set $k=\delta_T(A)$.
By Lemma 1 of \cite{Zhu2023},
$\mathcal{M}(T^{odd})$ contains a matching for any nontrivial tree $T$. Additionally,
by Theorem \ref{thm2.2}, there exists a graph $G^\star\in \mathcal{H}(n,2,\gamma(\mathcal{M}(T^{odd})),\mathcal{B}(T^{odd}),K_{k-1,k-1})$,
which is also an extremal graph with respect to $\mathrm{ex}(n,T^{odd})$.
Let $\gamma(\mathcal{M}(T^{odd}))=\gamma$.
Recall the definition of $\mathcal{H}(n,2,\gamma,\mathcal{B}(T^{odd}),K_{k-1,k-1})$.
The graph $G^\star$ can be obtained from $Q\nabla T_{2}(n-\gamma+1)$
for some $Q\in {\rm EX}(\gamma-1,\mathcal{B}(T^{odd}))$
by embedding a copy of $K_{k-1,k-1}$ into one color class of $T_{2}(n-\gamma+1)$.
Since $T$ is fixed, $k$ is independent of $n$, and hence
the second factor differs from $T_2(n-\gamma+1)$ by
$(k-1)^2=O(1)$ edges.
Therefore, by Definition \ref{definition-1.2},
$\mathcal{M}(T^{odd})$ is matching-good.

(ii)
From the proof of Lemma 1 in \cite{Peng-2024},
we can see that $\gamma(\mathcal{M}(K_{s,t}^{odd}))=\beta(\mathcal{M}(K_{s,t}^{odd}))=s$ and $\mathcal{M}(K_{s,t}^{odd})$ contains a matching.
Since $\gamma(\mathcal{M}(K_{s,t}^{odd}))=\beta(\mathcal{M}(K_{s,t}^{odd}))=s$,
we know that $\mathcal{B}(K_{s,t}^{odd})=\{K_{s}\}$ and ${\rm EX}(s-1,\mathcal{B}(K_{s,t}^{odd}))=\{K_{s-1}\}$.
To prove that $\mathcal{M}(K_{s,t}^{odd})$ is matching-good,
by Definition \ref{definition-1.2},
it suffices to show that there exists a graph $G^\star\in {\rm EX}(n,K_{s,t}^{odd})$ such that
$G^\star=K_{s-1}\nabla H_2$,
where $H_2$ differs from $T_{2}(n-s+1)$ by $O(1)$ edges.

Set $G^\star=H'(n,2,2)$ if $(s,t)=(2,2)$, and $G^\star=G_{s,t}$ otherwise.
Then, by Theorem \ref{thm2.3}, we know that
$G^\star\in {\rm EX}(n,K_{s,t}^{odd})$ for sufficiently large $n$.
In particular, when $(s,t)=(2,2)$,
we have $G^\star=K_1\nabla H_2$,
where $H_2$ is obtained from $T_2(n-1)$ by adding exactly one edge.
Thus, $H_2$ differs from $T_2(n-1)$ by exactly one edge, which is $O(1)$.
In the remaining cases,
$G^\star=K_{s-1}\nabla H_2$,
where $H_2$ is obtained from $T_2(n-s+1)$ by embedding a copy of
$H_{t-1,t-1}$ into one of its color classes.
By the definition of $H_{t-1,t-1}$,
this adds exactly $e(H_{t-1,t-1})=(t-1)^2-(t-2)$ edges,
which is $O(1)$ since $t$ is fixed.
Therefore, by Definition \ref{definition-1.2},
$\mathcal{M}(K_{s,t}^{odd})$ is matching-good.
\end{proof}

Given a graph $H$ and an integer $p\geq 2$, recall that the edge blow-up of $H$, denoted by $H^{p+1}$,
is the graph obtained by replacing each edge in $H$ by a $K_{p+1}$ where the new vertices of $K_{p+1}$ are all distinct.

\begin{proof}[\textbf{Proof of Theorem \ref{thm1.13C}}]
Fix a graph $H$ and an integer $p\geq \chi(H)+1$.
For brevity, we write $\gamma(\mathcal{M}(H^{p+1}))=\gamma$,  $\beta(\mathcal{M}(H^{p+1}))=\beta$,
and $\mathcal{B}(H^{p+1})=\mathcal{B}$.
Clearly, $p(H^{p+1})=\chi(H^{p+1})-1=p\geq 2$.
From Lemma 12 of \cite{L2013} we know that the family
$\mathcal{M}(H^{p+1})$ contains a matching.
To prove that $\mathcal{M}(H^{p+1})$ is matching-good,
by Definition \ref{definition-1.2},
it suffices to show that there exists a graph $G^\star\in {\rm EX}(n,H^{p+1})$
such that $G^{\star}=Q\nabla H_2$,
where $Q\in {\rm EX}(\gamma-1,\mathcal{B})$ and $H_2$ differs from $T_{p}(n-\gamma+1)$ by $O(1)$ edges.

(i) and (ii):
Let $H$ be either a non-bipartite graph, or a bipartite graph with $\gamma<\gamma(H)$.
Theorem \ref{thm1.11C} implies that $G^\star\in \mathcal{H}(n,p,\gamma,\mathcal{B})$.
By the definition of $\mathcal{H}(n,p,\gamma,\mathcal{B})$,
it follows that $G^\star\cong Q\nabla T_{p}(n-\gamma+1)$
for some $Q\in {\rm EX}(\gamma-1,\mathcal{B})$.
Thus, $\mathcal{M}(H^{p+1})$ is matching-good.

(iii)
Let $H$ be an even cycle $C_{2k}$ with $k\geq 2$.
Clearly, $\beta=\beta(C_{2k})=k$ and $\gamma(C_{2k})=k$.
By the definition of $\gamma$, it follows that $\gamma\leq \gamma(C_{2k})$.
Combining these observations with \eqref{align-001}, we obtain
\begin{align*}
k=\beta(C_{2k})=\beta \leq \gamma\leq \gamma(C_{2k})=k,
\end{align*}
which yields that  $\gamma=\beta=k$.
Hence, ${\rm EX}(\gamma-1,\mathcal{B})=\{K_{\gamma-1}\}$.
By Theorem 3 of \cite{L2013}, for every even cycle
$C_{2k}$ with $k\geq 2$, including $C_4$,
there exists an extremal graph $G^\star$
obtained from
$K_{k-1}\nabla T_p(n-k+1)$
by adding exactly one edge within one
color class of $T_p(n-k+1)$,
for sufficiently large $n$.
Since $\gamma=k$, the second factor
differs from $T_p(n-\gamma+1)$
by exactly one edge, which is $O(1)$.
Therefore, by Definition \ref{definition-1.2},
$\mathcal{M}(H^{p+1})$ is matching-good.
\end{proof}

\section{Concluding remarks}

By Lemma \ref{LEMMA2.4},
if $\mathcal{M}(\mathcal{H})$ contains a linear forest,
then there exists $\varphi=\varphi(\mathcal{H})$ such that
$\mathcal{D}(n,p,\varphi)$ contains a graph $G\in \mathrm{EX}(n,\mathcal{H})$
for sufficiently large $n$.
It is natural to ask whether the spectral analogue still holds, namely:

\begin{prob}\label{prob6.1}
Let $\mathcal{H}$ be a finite graph family
such that $p(\mathcal{H})=p\geq 2$ and $\mathcal{M}(\mathcal{H})$ contains a linear forest.
Does there exist a constant
$\varphi=\varphi(\mathcal{H})$ such that
$\mathcal{D}(n,p,\varphi)\cap \mathrm{SPEX}(n,\mathcal{H})\neq\varnothing$
for all sufficiently large $n$?

\end{prob}

Clearly, a matching is a special case of a linear forest.
Thus, Lemmas \ref{LEMMA2.2} and \ref{LEMMA2.3} imply that
Problem \ref{prob6.1} has an affirmative answer
when $\mathcal{M}(\mathcal{H})$ contains a matching.
Theorem \ref{thm1.1} adds that
if $\mathcal{M}(\mathcal{H})$ is matching-good,
then $\mathcal{H}$ is spectral-consistent.
While this covers many non-bipartite graphs,
there remain cases whose decomposition family contains a linear forest
but no matching.
One notable example is the icosahedron $I^{12}$:
Its edge-extremal graph is already known (see \cite[Theorem 2.5]{Simonovits1974});
however, the corresponding spectral-extremal problem remains open.

\section*{Acknowledgements}

Longfei Fang and Huiqiu Lin were supported by the National Natural Science Foundation of China (Nos.~12501471 and 12671406).
Sergey Goryainov was supported by the Natural Science Foundation of Hebei Province (A2023205045) and the 111 Center (Grant No.D26018).
Denis Krotov was supported by the grant of the Natural Science Foundation of Hebei Province (Project No. A2023205045), by the 111 Center (Grant No. D26018),
and within the framework of the state contract of the Sobolev Institute of Mathematics (FWNF-2026-0011).


\begin{thebibliography}{99}
\setlength{\itemsep}{0pt}

%\bibitem {AHS1972}
%H.L. Abbott, D. Hanson, N. Sauer, Intersection theorems for systems of sets, \emph{J. Combin. Theory Ser. A} \textbf{12} (1972) 381--389.

\bibitem {Chen2025}
G.T. Chen, X.Y. Lei, S.C. Li,
The exact Tur\'{a}n number of disjoint graphs--A generalization of Simonovits' theorem, and beyond,
\emph{European J. Combin.} \textbf{130} (2025) Paper No. 104226.


\bibitem {Chi2023}
C. Chi, L.-T. Yuan,
The Tur\'{a}n number for the edge blow-up of trees: the missing case,
\emph{Discrete Math.} \textbf{346} (2023), no. 6, Paper No. 113370, 8 pp.


\bibitem {Cioaba2}
S. Cioab\u{a}, D.N. Desai, M. Tait, The spectral radius of graphs with no odd wheels,
\emph{European J. Combin.} \textbf{99} (2022), Paper No. 103420, 19 pp.


\bibitem {CFTZ2020}
S. Cioab\u{a}, L.H. Feng, M. Tait, X.-D. Zhang,
The maximum spectral radius of graphs without friendship subgraphs,
\emph{Electron. J. Combin.} \textbf{27} (2020), no. 4, Paper No. 4.22, 19 pp.





\bibitem {DKL2022}
D.N. Desai, L.Y. Kang, Y.T. Li, Z.Y. Ni, M. Tait, J. Wang, Spectral extremal graphs for intersecting cliques,
\emph{Linear Algebra Appl.} \textbf{644} (2022) 234--258.

%\bibitem {Erdos-1967}
%P. Erd\H{o}s,  Some recent results on extremal problems in graph theory (Results),
%In: Theory of Graphs (International Symposium Rome, 1966), Gordon and
%Breach, New York, Dunod, Paris, 1966, pp. 117--123.

%\bibitem {Erdos-1968}
%P. Erd\H{o}s, On some new inequalities concerning extremal properties of
%graphs, In: Theory of Graphs (Proceedings of the Colloquium, Tihany, 1966),
%Academic Press, New York, 1968, pp. 77--81.

%\bibitem {Simonovits-1968}
%M. Simonovits, A method for solving extremal problems in graph theory,
%stability problems, In: Theory of Graphs (Proceedings of the Colloquium,
%Tihany, 1966), Academic Press, New York, 1968, pp. 279--319.

%\bibitem {E1962}
%P. Erd\H{o}s, \"{U}ber ein Extremalproblem in der Graphentheorie (German),
%\emph{Arch. Math. (Basel)} \textbf{13} (1962) 122--127.

\bibitem {EFGG1995}
P. Erd\H{o}s, Z. F\"{u}redi, R.J. Gould, D.S. Gunderson, Extremal graphs for intersecting triangles,
\emph{J. Combin. Theory Ser. B} \textbf{64} (1995), no. 1, 89--100.

\bibitem {EG1959}
P. Erd\H{o}s, T. Gallai, On maximal paths and circuits of graphs, \emph{Acta Math. Acad. Sci. Hungar.}
\textbf{10} (1959) 337--356.


\bibitem {Erdos-1966} P. Erd\H{o}s, M. Simonovits, A limit theorem in graph theory,
\emph{Studia Sci. Math. Hungar.} \textbf{1} (1966) 51--57.

\bibitem {ES1946}
P. Erd\H{o}s, A.H. Stone,
On the structure of linear graphs,
\emph{Bull. Amer. Math. Soc.}
\textbf{52} (1946) 1087--1091.

\bibitem {Fang2026}
L.F. Fang, H.Q. Lin,
Spectral extremal results on edge blow-up of graphs,
\emph{Discrete Math.} \textbf{349} (2026), no. 2, Paper No. 114835, 18 pp.

\bibitem {Fang-2025}
L.F. Fang, M. Tait, M.Q. Zhai,
Decomposition family and spectral extremal problems on non-bipartite graphs,
\emph{Discrete Math.} \textbf{348} (2025), no. 10, Paper No. 114527, 16 pp.

\bibitem {Fang-2024}
L.F. Fang, M.Q. Zhai, H.Q. Lin,
Spectral extremal problem on $t$ copies of $\ell$-cycle,
\emph{Electron. J. Combin.} \textbf{31} (2024), no. 4, Paper No. 4.17, 30 pp.


\bibitem {FS1975}
R.J. Faudree, R.H. Schelp, Path Ramsey numbers in multicolorings,
\emph{J. Combin. Theory Ser. B} \textbf{19} (1975), no. 2, 150--160.

\bibitem {Fiedler2010}
M. Fiedler, V. Nikiforov,
Spectral radius and Hamiltonicity of graphs,
\emph{Linear Algebra Appl.} \textbf{432} (2010), no. 9, 2170--2173.

\bibitem {FG2015}
Z. F\"{u}redi, D.S. Gunderson, Extremal numbers for odd cycles,
\emph{Combin. Probab. Comput.} \textbf{24} (2015), no. 4, 641--645.

\bibitem {FS2013}
Z. F\"{u}redi, M. Simonovits,
The history of degenerate (bipartite) extremal graph problems, Erd\H{o}s centennial, 169--264,
Bolyai Soc. Math. Stud., 25, J\'{a}nos Bolyai Math. Soc., Budapest, 2013.

\bibitem {Ge2020}
J. Ge, B. Ning,
Spectral radius and Hamiltonian properties of graphs, II,
\emph{Linear Multilinear Algebra} \textbf{68} (2020), no. 11, 2298--2315.

\bibitem {Guiduli-1996}
B.D. Guiduli, Spectral extrema for graphs, Ph.D. Thesis, 105 pages,
University of Chicago, December 1996.


\bibitem {Hou-2024}
J.F. Hou, H. Li, Q.H. Zeng,
Extremal graphs for the suspension of edge-critical graphs,
\emph{Electron. J. Combin.} \textbf{31} (2024), no. 4, Paper No. 4.55, 15 pp.


\bibitem{Konig-1931}
D. K\"{o}nig, Graphs and matrices,
\emph{Mat. Fiz. Lapok.} \textbf{38} (1931) 116--119.


\bibitem {Lei-2024}
X.Y. Lei, S.C. Li,
Spectral extremal problem on disjoint color-critical graphs,
\emph{Electron. J. Combin.} \textbf{31} (2024), no. 1, Paper No. 1.25, 19 pp.


\bibitem{Li2023}
B.L. Li, B. Ning,
Stability of Woodall's theorem and spectral conditions for large cycles,
 \emph{Electron. J. Combin.} \textbf{30} (2023), no. 1, Paper No. 1.39, 20 pp.


\bibitem {Li}
Y.T. Li, Y.J. Peng, The spectral radius of graphs with no intersecting odd cycles,
\emph{Discrete Math.} \textbf{345} (2022), no. 8, Paper No. 112907, 16 pp.

\bibitem {Lin2022}
H.Q. Lin, M.Q. Zhai, Y.H. Zhao, Spectral radius, edge-disjoint cycles and cycles of the
 same length, \emph{Electron. J. Combin.} \textbf{29} (2022), no. 2, Paper No. 2.1, 26 pp.

\bibitem {L2013}
H. Liu, Extremal graphs for blow-ups of cycles and trees, \emph{Electron. J. Combin.} \textbf{20} (2013), no. 1, Paper No. 65, 16 pp.

\bibitem {Liu2023+}
L.L. Liu, B. Ning,
Spectral Tur\'{a}n-type problems on sparse spanning graphs,
\emph{Discrete Math.} \textbf{349} (2026), no. 7, Paper No. 115016.

\bibitem {Ma2025}
J. Ma, L.-T. Yuan,
Supersaturation beyond color-critical graphs,
\emph{Combinatorica} \textbf{45} (2025), no. 2, Paper No. 18, 40 pp.

\bibitem {M1968}
J.W. Moon, On independent complete subgraphs in a graph,
\emph{Can. J. Math.} \textbf{20} (1968) 95--102.

\bibitem {NWK2023}
Z.Y. Ni, J. Wang, L.Y. Kang, Spectral extremal graphs for disjoint cliques,
\emph{Electron. J. Combin.} \textbf{30} (2023), no. 1, Paper No. 1.20, 16 pp.

\bibitem {NWK2023+}
Z.Y. Ni, J. Wang, L.Y. Kang,
Extremal problems for disjoint graphs,
arXiv:2308.07608.

\bibitem {Nikiforov5}
V. Nikiforov, Bounds on graph eigenvalues II,
\emph{Linear Algebra Appl.} \textbf{427} (2007), nos. 2--3, 183--189.

\bibitem {Nikiforov10}
V. Nikiforov, Spectral saturation: inverting the spectral Tur\'{a}n theorem,
\emph{Electron. J. Combin.} \textbf{16} (2009), no. 1, Research Paper 33, 9 pp.


\bibitem {Peng-2024}
X. Peng, M.J. Xia,
Tur\'{a}n number of the odd-ballooning of complete bipartite graphs,
\emph{J. Graph Theory} \textbf{107} (2024), no. 1, 181--199.

\bibitem {Simonovits1966}
M. Simonovits, A method for solving extremal problems in graph theory, stability problems,
in: Theory of Graphs (Proc. Colloq., Tihany, 1966), Academic Press, New York,
1968, pp. 279--319.

\bibitem {Simonovits1974}
M. Simonovits, Extremal graph problems with symmetrical extremal graphs. Additional chromatic conditions,
\emph{Discrete Math.} \textbf{7} (1974) 349--376.


\bibitem{Simonovits1999}
M. Simonovits,
How to solve a Tur\'{a}n type extremal graph problem? (linear decomposition),
in: \emph{Contemporary Trends in Discrete Mathematics},
DIMACS Ser. Discrete Math. Theoret. Comput. Sci., Vol. 49,
Amer. Math. Soc., Providence, RI, 1999, pp. 283--305.



\bibitem {Wang2021}
A.Y. Wang, X.M. Hou, B.Y. Liu, Y. Ma,
The Tur\'{a}n number for the edge blow-up of trees,
\emph{Discrete Math.} \textbf{344} (2021), no. 12, Paper No. 112627, 12 pp.

\bibitem {WANG-2023}
J. Wang, L.Y. Kang, Y.S. Xue, On a conjecture of spectral extremal problems,
\emph{J. Combin. Theory Ser. B} \textbf{159} (2023) 20--41.

\bibitem {WANG-2024}
J. Wang, Z.Y. Ni, L.Y. Kang, Y.-Z. Fan,
Spectral extremal graphs for edge blow-up of star forests,
\emph{Discrete Math.} \textbf{347} (2024), no. 10, Paper No. 114141, 14 pp.


\bibitem {Yan-2021}
N. Yan,
The Tur\'{a}n number of graphs with given decomposition family,
\emph{Acta Scientiarum Naturalium Universitatis Nankaiensis}
\textbf{54} (2021), no. 4, 34--43.


\bibitem {Y2022}
L.-T. Yuan, Extremal graphs for edge blow-up of graphs,
\emph{J. Combin. Theory Ser. B} \textbf{152} (2022) 379--398.


\bibitem {ZL2023}
M.Q. Zhai, H.Q. Lin, A strengthening of the spectral chromatic critical edge theorem: books and theta graphs,
\emph{J. Graph Theory} \textbf{102} (2023), no. 3, 502--520.

\bibitem {Zhai2022}
M.Q. Zhai, R.F. Liu, J. Xue,
A unique characterization of spectral extrema for friendship graphs,
\emph{Electron. J. Combin.} \textbf{29} (2022), no. 3, Paper No. 3.32, 17 pp.


%\bibitem {Zhang2024+}
%W.Q. Zhang,
%Walks, infinite series and spectral radius of graphs,
%arXiv:2406.07821v3.

\bibitem {Zhang2025+}
W.Q. Zhang,
Spectral skeletons and applications,
\emph{Discrete Math.} \textbf{349} (2026), no. 8, Paper No. 115110.


\bibitem {Zhu2020}
H. Zhu, L.Y. Kang, E.F. Shan,
Extremal graphs for odd-ballooning of paths and cycles,
\emph{Graphs Combin.} \textbf{36} (2020), no. 3, 755--766.

\bibitem {Zhu2023}
X.T. Zhu, Y.J. Chen,
Tur\'{a}n number for odd-ballooning of trees,
\emph{J. Graph Theory} \textbf{104} (2023), no. 2, 261--274.

\end{thebibliography}
\end{document}